%% file: paper.tex
\documentclass[reqno,11pt,a4paper]{amsart}

\usepackage{verbatim}
\usepackage{tikz}
\usetikzlibrary{matrix,arrows,calc,snakes,patterns,decorations.markings, arrows.meta}
\usepackage{soul}

\usepackage[utf8]{inputenc} 
\usepackage[T1]{fontenc}

\usepackage[mathscr]{eucal}
\usepackage{graphics,epic}
\usepackage{amsfonts, mathtools}
\usepackage{amscd}
\usepackage{latexsym}
\usepackage{amsmath,amssymb, amsthm, stmaryrd, bm, bbm}
\usepackage[all,2cell]{xy}
\usepackage{mathrsfs}
\usepackage{url, hyperref}
\usepackage[normalem]{ulem}

\hypersetup{
    colorlinks,
    linkcolor={red!50!black},
    citecolor={blue!50!black},
    urlcolor={blue!80!black}
}

\input{macros} %Macros Annette

\input{macros-mk} %Macros Martin

\begin{document}

\title{Dimension formulas for period spaces via motives and species}

\author{Annette Huber} 
\email{\texorpdfstring{annette.huber@math.uni-freiburg.de}{}}

\author{Martin Kalck}
\email{\texorpdfstring{martin.kalck@uni-graz.at}{}}
\date{\today}
\begin{abstract}
We apply the structure theory of finite dimensional algebras to deduce dimension formulas for spaces of period numbers, i.e., complex numbers defined by integrals of algebraic nature. We get a complete and conceptually clear answer in the case of $1$-periods, generalising classical results like Baker's theorem on the logarithms of algebraic numbers and completing partial results in Huber--W{\"u}stholz \cite{huber-wuestholz}. 

The application to the case of Mixed Tate Motives over $\Z$ (i.e., Multiple Zeta Values) recovers the dimension estimates of Deligne--Goncharov \cite{deligne-goncharov} for the space of multiple zeta values of a given weight.

\end{abstract}

\maketitle

\tableofcontents
\input{intro} %Einleitung

\input{kapitel1} % Basic Facts on Species and results from the literature

\input{kapitel2b} %alternativer Aufbau
\input{kapitel2c} %Dimensionsformeln
\input{kapitel3} %Saturation

\input{kapitel4} %Perioden und Dimensionsformeln für Motive, allgemein

\input{kapitel5} %Fall von 1-Motiven

\input{kapitel6} %Mixed Tate Motives
\begin{appendix}
\input{anhang} % verträgliche Wahlen
\end{appendix}

\bibliographystyle{alpha}
\bibliography{periods,mk}

%\bibliography{periods}
%\bibliography{mk}

\end{document}

%% file: macros.tex
\newtheorem{lemma}{Lemma}[section]
\newtheorem{prop}[lemma]{Proposition}
\newtheorem{thm}[lemma]{Theorem}
\newtheorem{cor}[lemma]{Corollary}

\newtheorem{defnthm}[lemma]{Definition/Theorem}

\newtheorem{lem}[lemma]{Lemma}
\newtheorem*{ThmIntro}{Theorem}
\newtheorem*{CorIntro}{Corollary}

\theoremstyle{definition}
\newtheorem{defn}[lemma]{Definition}

\newtheorem{ex}[lemma]{Example}

\newtheorem{rem}[lemma]{Remark}
\newtheorem*{remIntro}{Remark}

\newtheorem{notation}[lemma]{Notation}
\newtheorem{setup}[lemma]{Set-Up}

\newcommand{\categoryfont}{\mathsf}
\newcommand{\algebrafont}{\mathit}
\newcommand{\modules}[1]{\operatorname{\algebrafont{#1}\categoryfont{-mod}}}  

\newcommand{\End}{\mathrm{End}}
\newcommand{\Hom}{\mathrm{Hom}}
\newcommand{\Ext}{\mathrm{Ext}}
\newcommand{\MM}{\mathcal{MM}}

\newcommand{\MTM}{\mathsf{MTM}}

\newcommand{\Ah}{\mathcal{A}}
\newcommand{\Bh}{\mathcal{B}}
\newcommand{\Ch}{\mathcal{C}}
\newcommand{\Sh}{\cs}
\newcommand{\isom}{\cong}

\newcommand{\ohne}{\smallsetminus}
\newcommand{\tensor}{\otimes}

\newcommand{\im}{\mathrm{Im}}

\newcommand{\rk}{\mathrm{rk}}
\newcommand{\dR}{\mathrm{dR}}
\newcommand{\DM}{\mathsf{DM}}

\newcommand{\gm}{\mathrm{gm}}
\newcommand{\DMgm}{\DM_{\gm}}
\newcommand{\onemot}{1\mathrm{-Mot}_\Qbar}

\newcommand{\sing}{\mathrm{sing}}
\newcommand{\Per}{\mathcal{P}} %Perioden 
\newcommand{\Pertilde}{\widetilde{\Per}}

\newcommand{\Vsing}[1]{V_\sing(#1)}
\newcommand{\hsing}{H_\sing}
\newcommand{\hdR}{H_\dR}

\newcommand{\tate}{\mathrm{Ta}}

\newcommand{\alg}{\mathrm{alg}}
\newcommand{\inc}{\mathrm{inc2}}

\newcommand{\mix}{\mathrm{inc3}}
\newcommand{\sat}{\mathrm{sat}}

\newcommand{\HW}{\mathrm{HW}}
\newcommand{\mot}{\mathrm{mot}}
\newcommand{\MZV}{\mathsf{MZV}}
\newcommand{\even}{\mathrm{ev}}
\newcommand{\odd}{\mathrm{odd}}

\newcommand{\Q}{\mathbb{Q}}
\newcommand{\Qbar}{{\overline{\Q}}}
\newcommand{\Z}{\mathbb{Z}}
\newcommand{\C}{\mathbb{C}}
\newcommand{\G}{\mathbb{G}}
\newcommand{\Oh}{\mathcal{O}}
\newcommand{\Ih}{\ci}
\newcommand{\rr}{\mathfrak{r}}

\renewcommand{\ss}{\mathit{ss}}

\newcommand{\Gm}{\G_m}
\renewcommand{\ul}[1]{\underline{#1}}

\newcommand{\pfad}{\rightsquigarrow}
\newcommand{\tr}{\mathrm{tr}}
\newcommand{\gammacirc}{\overset{\circ}{\gamma}}

%% file: macros-mk.tex
\newcommand{\opname}[1]{\operatorname{\mathsf{#1}}}

\newcommand{\op}{^{op}}

\newcommand{\gldim}{\opname{gldim}\nolimits}
\newcommand{\rad}{\opname{rad}\nolimits}

\newcommand{\Char}{\mathrm{char}}

\newcommand{\N}{{\mathbb N}}

\newcommand{\cc}{{\mathcal C}}
\newcommand{\ci}{{\mathcal I}}

\newcommand{\cs}{{\mathcal S}}
\newcommand{\csmult}{\cs^{(\mathrm{mult})}}

\newcommand{\Tor}{\opname{Tor}}

%% file: intro.tex
\section*{Introduction}

Periods are complex numbers of the form
\[ \int_G\omega\]
where both $\omega$ and $G$ are of algebraic and even number theoretic nature. Examples include $\log(\alpha)$ for $\alpha\in\Qbar^*$ or the periods of elliptic curves defined over $\Qbar$. They are a classical object of study in transcendence theory. Periods are known to have a more conceptual interpretation in terms of motives, see Section~\ref{sec:periods} for more details. Our main result is a completely general upper bound for the $\Qbar$-dimension of the space $\Per\langle M\rangle$ of periods of a mixed motive $M$ over $\Qbar$. 

The result is of particular interest in the case of $1$-motives (corresponding to period integrals where $\omega$ is a $1$-form), complementing the dimension formulas in \cite[Part Four]{huber-wuestholz}. The main new insight is the conceptual description of the mysterious ``error term'' as the dimension of a space of
Yoneda $2$-extensions.

The precise statement needs some notation. Let $F$ be an algebraic field extension of $\Q$. Given a motive $M$ over $F$ let $\langle M\rangle $ be the full abelian subcategory generated by $M$, closed under subquotients inside the abelian category of all  motives.  These categories are $\Q$-linear, but their period spaces have natural $F$-vector space structures. Let $M_1,\dots,M_n$ be representatives of the isomorphism classes of the simple objects in $\langle M\rangle$, i.e., the simple subquotients of $M$. Let
$D_i=\End(M_i)$ (a finite dimensional divison algebra over $\Q$), $d_i=\dim_\Q D_i$, $m_i=\dim_{D_i}H_\sing(M_i)$
and $\Sh_M$ the species of $\langle M\rangle$, see Notation~\ref{not:mult}.

\begin{ThmIntro}\label{thm:intro}
In the above notation,  
\begin{multline}\label{eq:intro}
\dim_F\Per\langle M\rangle\leq \sum_{i=1}^n m^2_id_i + \sum_{i,j=1}^n\sum_{\gamma:i\pfad j}m_im_j\dim_\Q E(\gamma)
\\
-\sum_{i,j=1}^nm_im_j\sum_{k=2}^\infty \dim_\Q  \Ext^k_{\langle M\rangle}(M_i,M_j)
\end{multline}
(where the sum is taken over all paths $\gamma$ of positive length in the species $\cs_M$)
and
\[ \dim_\Q E(\gamma)=\frac{\prod\limits_{\text{$\varepsilon:v\to w$ on $\gamma$}}\dim_\Q \Ext^1_{\langle M\rangle}(M_w,M_v)}{\prod\limits_{v\in\gammacirc} d_v}\]
(where $\gammacirc$ is the set of vertices on the path which are different from the end points).

If the Period Conjecture holds for $M$, then we have equality in the formula \eqref{eq:intro} in the two cases
\begin{itemize}
\item $\langle M\rangle$ is hereditary, i.e., all $\Ext^k_{\langle M\rangle}$ vanish for $k\geq 2$; or
\item the maximal length of a path in the quiver of the species of $\langle M\rangle$ is $2$, in particular $\Ext^k_{\langle M\rangle}$ vanishes for $k\geq 3$. 
\end{itemize}
\end{ThmIntro}

For $1$-motives, the maximal length of paths in the quiver of the species is $2$ and the \emph{Period Conjecture} is known by \cite{huber-wuestholz}. In particular, the Theorem above has the following consequence.

\begin{CorIntro}
Let $M$ be a $1$-motive over $F=\Qbar$. Then we have equality in the formula \eqref{eq:intro}.

%In particular, 
\end{CorIntro}
The terms in the formula have an explicit interpretation in terms of the $1$-motive, see Theorem~\ref{thm:main1}.

\begin{remIntro}For general motives $M$, the estimate \eqref{eq:intro} is not sharp, see Example~\ref{Ex:dimEstimateLongerPaths} and Remark \ref{R:MixedTateArbitrary}.
\end{remIntro}

%The \emph{Period Conjecture}  is known in the case of $1$-motives by \cite{huber-wuestholz} and, indeed, the maximal length of a path in the quiver of the species is $2$ in this case. We refer to the detailed discussion in Chapter~\ref{sec:1-mot}, in particular Theorem~\ref{thm:main1}.

Let us explain how these dimension formulas tie in with the Period Conjecture. It makes a qualitative prediction about all relations between period numbers. In the approach of Kontsevich (see \cite{kontsevich})
and Nori (unpublished, but see \cite{period-buch} and \cite{huber_galois}) all $F$-linear relations are expected to be induced by functoriality of motives. This approach fits well with the classical results in transcendence theory obtained e.g. by Baker, Siegel and Wüstholz. Indeed, the Period Conjecture for $1$-motives is deduced from Wüstholz's Analytic Subgroup Theorem in
\cite{huber-wuestholz}. The qualitative statement can be translated into a quantitative one. We have unconditionally
\begin{equation}\label{eq:intro2} \dim_F\Per\langle M\rangle\leq \dim_\Q \End(H_\sing|_{\langle M\rangle}).\end{equation}
The Period Conjecture holds for $\Per\langle M\rangle$ if and only if we have equality in \eqref{eq:intro2}.

In general the version of the Period Conjecture is wide open, but we get explicit upper bounds. In particular, we recover the estimates of Terasoma, Deligne--Goncharov and Brown in the case of Mixed Tate Motives and their periods, e.g. Multiple Zeta Values, see Section~\ref{ssec:MTMZ}. Indeed, the category of all Mixed Tate Motives over $\Oh_F[S^{-1}]$ (where $F$ is a number field and $S$ a finite set of places) is hereditary. If a Mixed Tate Motive
$M$ is such that $\langle M\rangle$ is extension closed, the subcategory is hereditary as well and by the above Theorem we have equality in \eqref{eq:intro} if and only if the period conjecture holds for $M$.

\subsection*{Structure theory of finite dimensional algebras}
Summing up, the dimension of the space of periods is conjecturally equal to the $\Q$-dimension of the finite dimensional associative $\Q$-algebra that Nori attaches to a motive $M$. This is where the structure theory of such algebras comes in. A simple example of a finite dimensional algebra is the path algebra of a finite quiver (directed graph) without oriented cycles. Over algebraically closed fields, these are precisely the finite dimensional, basic, hereditary algebras. However, in our application there are three possible complications (not apparent in the best studied case of Mixed Tate Motives):
\begin{itemize}
\item The algebra is not hereditary.
\item The algebra is not basic, so that multiplicities have to be taken into account.
\item The ground field is $\Q$---so not algebraically closed. In this case, we have to consider path algebras attached not only to quivers but to \emph{species} as introduced by Gabriel, see \cite{Gabriel2}.
\end{itemize}
It is well-known that every basic finite dimensional algebra $B$ over a perfect field $k$
is the quotient of a (hereditary) path algebra $k\cs$ of a species by an admissible ideal $I$, cf. e.g. 
\cite[Section 4.1.]{Benson}, \cite[Section 8]{DrozdKiricenko}, \cite{Berg}.
%of the category $\modules{B}$. 
\[ B \cong k\cs/I. \]
Building on work of Bongartz \cite{Bongartz}, we describe the two-sided ideal $I$ in terms of $\Ext$-groups, see Theorem~\ref{T:Bongartz} and Theorem~\ref{thm:generate_ideal}, cf. also \cite{ButlerKing}. 

This allows us to deduce upper bounds for the $k$-dimension of $A$ from data of
$\modules{A}$. They are sharp if the maximal length of a path in the quiver of $A$ is at most $2$. This covers the case of $1$-motives. 

A second application is the construction of a \emph{hereditary closure} of
a $k$-linear abelian category inside a hereditary category. Conjecturally the category of mixed motives over $\Qbar$ is in fact hereditary. By replacing a motive by its hereditary closure (or saturation), we can reduce all considerations of relations between periods to the hereditary case.

\subsection*{Structure of the paper}
The first half of the paper treats abstract algebra.
The first chapter reviews terminology and facts from the theory of finite dimensional algebras. It is aimed at readers with a background outside of representation theory.

Section~\ref{sec:ext} is devoted to the generalisation of Bongartz's description of the defining ideal of a basic finite dimensional algebra inside the path algebra of its species to the case of perfect ground fields.

These insights are turned into dimension formulas in Section~\ref{sec:dim}. The hereditary closure is constructed in Section~\ref{sec3}.

In the second half of the paper, we turn to periods and explain the implications of the abstract results to motives in general in Section~\ref{sec:periods}, to $1$-motives in Section~\ref{sec:1-mot} and to Mixed Tate Motives in Section~\ref{sec:MTM}.

\subsection*{Acknowledgements}
We thank Wolfgang Soergel for discussions and comments and Javier Fresan for his many detailed comments and corrections.

We are particularly indebted to Francis Brown. This project started after a discussion of the first author with him at the Simons Symposium in 2022. He pointed out that his algorithm for  computing linear relations in \cite{brown-phys}
 works not only for Mixed Tate Motives, but in fact for any hereditary category of motives and thus the crucial role of this property. He also stressed that the radical (coradical in Tannakian language) should be more important than the weight filtration. The reader will find these two themes throughout our paper. We refer to Remark~\ref{rem:brown} for more on the connection.

The hereditary case of our formulas was treated by Hörmann in \cite{hoermann-notiz}, who also made the connection between the ``saturated case'' appearing of  \cite{huber-wuestholz} and the hereditarity condition. 

We are grateful to Julian K{\"u}lshammer for detailed comments and
thank Henning Krause and Sondre Kvamme for their interest and for comments on an earlier version of this text.

M.K. was partially funded by the Deutsche Forschungsgemeinschaft (DFG,
German Research Foundation) – Projektnummer 496500943. He takes no credit for the results about the `hereditary closure'.

%% file: kapitel1.tex
\section{Background on species and Morita-theory}

In this chapter, we collect well-known and maybe not so well-known facts about finite dimensional algebras. These facts will be translated into dimension formulas in Section~\ref{sec:dim}.

\subsection{Notations and conventions}\label{sec:notation}
Throughout this paper, we assume that $k$ is a perfect field. 

We will study \emph{finite dimensional $k$-algebras}, which we always assume to be associative and unital. Given such an algebra $A$, we denote its \emph{radical} by  $\rad(A)$ and let $K=A/\rad(A)$ be the semi-simple quotient. For an algebra $A$, unless specified otherwise we always consider finitely generated \emph{left} $A$-modules and denote the corresponding category by $\modules{A}$. We sometimes write ${}_AX$ to emphasise the left-action on a left-module $X$ and, correspondingly, $Y_A$ for the right-action on a right module $Y$.

 We use the notation $D(-)=\Hom_k(-, k)$. It is a functor
$\modules{A}\to\modules{A\op}$.

If $R$ and $S$ are $k$-algebras (not necessarily finite dimensional over $k$) an $R-S$-bimodule $M$ is an $R$-left module and $S$-right module such that the actions satisfy $(rm)s=r(ms)$ for all $m\in M, r\in R, s\in S$.

Let $\cc$ be an abelian category and let $M$ be an object in $\cc$. We denote the full subcategory of $\cc$ that contains all subquotients of $M^{\oplus n}$ for all $n \in \N$ by $\langle M \rangle$.

A \emph{quiver} is a directed graph $Q=(V,E,s,t)$ with sets of \emph{vertices} 
$V$, \emph{edges} $E$ and each edge $\varepsilon\in E$ is assigned an initial vertex $s(\varepsilon)$ and a terminal vertex $t(\varepsilon)$. It is called \emph{finite} if it has only finitely many vertices and edges. Most quivers in this paper will be finite. A priori, self-edges (loops) and multiple edges are allowed. 

A \emph{path} $\gamma=\varepsilon_1\cdots \varepsilon_n:v\pfad w$ is 
a sequence $\varepsilon_1,\dots,\varepsilon_n$ of edges such that $t(\varepsilon_1)=w$, $s(\varepsilon_n)=v$ and
$s(\varepsilon_i)=t(\varepsilon_{i+1})$:
\[ \gamma: v\xrightarrow{\varepsilon_n} v_n\xrightarrow{\varepsilon_{n-1}} v_{n-1}\xrightarrow{\varepsilon_{n-2}}\cdots \xrightarrow{\varepsilon_1} w.\]
We allow empty paths $\gamma:v\pfad v$ of length $0$. Let $P(Q)$ be the set of paths in $Q$. Paths can be concatenated as follows: if $\gamma:v\pfad w$ and $\delta:w\pfad u$ are paths, then  $\delta\gamma:v\pfad u$ is the concatenation.
A quiver is called \emph{directed} if it contains no oriented cycles, i.e., there are no closed paths of positive length. 

Given a quiver $Q$, its \emph{path algebra} $kQ$ is defined as
\[ kQ=\bigoplus_{\gamma\in P(Q)} k\gamma\]
where the composition is defined by concatenation of paths. (Here the product of two paths that can not be concatenated is set to $0$.)

We identify the vertices of $Q$ with paths of lengths $0$ and edges with paths of length $1$. The path algebra if finite dimensional if and only if $Q$ is finite and directed. 

For the related notion of a species and its path algebra, we refer to Section~\ref{sec:species}.

\subsection{Duals}
Throughout this section, we fix a finite dimensional $k$-algebra $A$ over
our perfect field $k$. We put $K=A/\rad(A)$.
Recall that $D:\modules{A}\to \modules{A\op}$ is defined as the $k$-dual.

\begin{lem}\label{lem:Ext_Tor}
For all $K$-modules $X$, the trace $\tr:K\to k$ induces a natural isomorphism
\[ DX\isom \Hom_K(X,K)\]
of $K\op$-modules. Moreover,
$\Tor_{m}^A(K, K) \cong D\Ext^{m}_A(K, K)$ as $K-K$-bimodules.
\end{lem}
\begin{proof}
 The first claim is an instance of \cite[Theorem 3.1]{Rickard}. Indeed, 
finite dimensional skew fields over perfect fields are symmetric via the trace map,
see \cite[VIII, p.~375, Cor.]{bourbaki8}. The property extends to matrix algebras over such skew fields and by Wedderburn's Theorem to $K$.

In particular, 
 \[ D(K_K)\isom {}_KK,\quad D({}_KK)\isom K_K\]
as $K$-modules and $K\op$-modules, respectively.

We now claim that 
\begin{equation}\label{ext-tor} DY\otimes_AX\cong D\Hom_A(X,Y)\end{equation}
for all finitely generated $A$-modules. There is a natural map from the left
to the right:
\[ f\otimes x\mapsto ( g\mapsto f(g(x)).\]
Both sides are right exact in $X$ as $D$ is exact. By taking a free resolution
of $X$, we are reduced to the case $X=A$. In this case the map is the 
\[ \mathrm{id}: DY\to DY.\]
The isomorphism \eqref{ext-tor} passes to derived functors, so that we have
\[ \Tor_n^A(DY,X)\isom D\Ext^n_A(X,Y).\]
The isomorphism is natural in both arguments.
We evaluate in $X=Y={}_KK$. The claim
follows because $DK\cong K\op$. It is $K-K$-equivariant by naturality.
\end{proof}

\begin{lemma}\label{lem:compute_ext}
There is a canonical isomorphism of $K-K$-bimodules
\begin{equation}\label{can} \rad(A)/\rad(A)^2\isom D\Ext^1_A(K,K).\end{equation}
\end{lemma}
\begin{proof}
The long exact sequence attached to the short exact sequence
\begin{equation}\label{eq:rr} 0\to \rad(A)\to A\to K\to 0\end{equation}
and the functor $\Hom_A(\cdot,K)$ induce a natural isomorphism
\[ \Ext^1_A(K,K)\isom \Hom_A(\rad(A),K)\isom \Hom_K(\rad(A)/\rad(A)^2,K)\]
because $\Hom_A(A,K)=\Hom_A(K,K)$ and $\Ext^1_A(A,K)=0$.
By Lemma~\ref{lem:Ext_Tor} we deduce
\[ \Ext^1_A(K,K)\isom D\rad(A)/\rad(A)^2.\]
Dually we get the claim.
\end{proof}

\subsection{Categories of modules}\label{ssec:modules}

The following result recalls 
different characterisations of the category of finitely generated modules $\modules{A}$ over a finite dimensional $k$-algebra $A$.

\begin{prop}\label{P:Deligne}
Let $\cc$ be a $k$-linear abelian category. The following conditions are equivalent:
\begin{itemize}
\item[(a)] $\cc \cong \modules{A}$, where $A$ is a finite dimensional $k$-algebra.\\
\item[(b)] All Hom-spaces in $\cc$ are finite dimensional over $k$ and $\cc$ has a projective generator $P$, i.e. every object is a quotient of $P^n$ for some $n \in \N$.\\
\item[(c)]  \begin{itemize}
\item[(i)] Every object in $\cc$ has finite length, i.e. all objects are both noetherian and artinian.
\item[(ii)] $\dim_k \Hom_\cc(X, Y) < \infty$ for all $X, Y \in \cc$.
\item[(iii)] We have $\langle G \rangle =\cc$ for some object $G \in \cc$.\\
\end{itemize}
\item[(d)]  
\begin{itemize}
\item[(i)] Every object in $\cc$ has finite length, i.e. all objects are both noetherian and artinian. 
\item[(ii)] There are only finitely many isomorphism classes of simple objects in $\cc$.
\item[(iii)] $\dim_k \Ext^1_\cc(S, T) < \infty$ for all simple objects $S, T \in \cc$. 
\item[(iv)] There exists an integer $l$ such that every object $X$ in $\cc$ has a filtration 
\[
0=X_0 \subset X_1 \subset X_2 \subset \cdots \subset X_l=X,
\]
such that all subfactors $X_i/X_{i-1}$ are semisimple.
\item[(v)] We have $\dim_k\End_\cc(S)<\infty$ for all objects $S \in \cc$, which are simple, projective and injective.  
\end{itemize}
\end{itemize}
\end{prop}
\begin{proof}
The equivalence between $(a)$ and $(b)$ is Morita theory, cf.~\cite[Thm. 2.11 \& Rem. 2.13]{Paquette}. The implication $(c) \Rightarrow (a)$ is \cite[Cor.~2.17]{deligne-festschrift}. For the implication $(a) \Rightarrow (c)$, we can take $G=A$, the free module. 

The implication from $(a)$ to $(d)$ is standard.
For the converse, we use \cite{Gabriel2}. More precisely, by
\cite[7.2]{Gabriel2}, the condition (i) implies that $\cc$ is the category of finite length modules over a pseudo-compact ring $A$. Conditions (ii) -- (iv) and
\cite[8.2]{Gabriel2} imply that $A$ is even Artinian. Hence, ${}_AA$ has finite length by the Hopkins--Levitzki Theorem and $\cc$ has a projective generator $P$ with $\End_\cc(P)\cong A$. It remains to show that $A$ is finite dimensional.  We claim that $\dim_k\End_\cc(S)<\infty$ for all simple objects $S$ in $\cc$. 
Then, by induction on the length, all $\Hom_\cc$ are finite dimensional, in particular, $\End_\cc(P)\isom A$ is finite dimensional over $k$.

We prove the claim. If $S$ is projective and injective this is (v), otherwise, without loss of generality, $ \Ext^1_\cc(S, T) \neq 0$ for a simple object $T$. Since this is a module over the $k$-skew-field $\End_\cc(S)$, we have  $\dim_k \End_\cc(S) < \dim_k \Ext^1_\cc(S, T)$. Now the claim follows from (iii).
\end{proof}

There does not seem to be a standard name for such categories in the literature. We introduce:
\begin{defn}\label{defn:str_finitary} 
A category satisfying the equivalent conditions of Proposition~\ref{P:Deligne} is called \emph{strongly finitary}. 
\end{defn}

\begin{cor}\label{cor:is_str_finitary}
Let $\cc$ be a $k$-linear abelian category such that $\Hom_\cc(X, Y)$ is finite dimensional for all $X, Y \in \cc$. Let $M \in \cc$ be an object of finite length. Then there is a finite dimensional algebra $A$ such that 
\[
\langle M \rangle \cong \modules{A}.
\]
\end{cor}

\begin{rem}\label{R:relations}
Let $A$ be a $k$-algebra, not necessarily finite dimensional and let $V$ be a finite dimensional representation of $A$. 
In particular, there is a $k$-algebra homomorphism
\[
\varphi\colon A \to \End_k(V) 
\] 
with kernel a two-sided ideal $\Ih:=\ker \varphi$. If $\rho \in \Ih$, we say that $V$ satisfies the relation $\rho$.
Then every object in the abelian subcategory $\langle V \rangle$ generated by $V$ also satisfies all relations $\rho \in \Ih$.

Indeed, the subcategory of $\modules{A}$ satisfying relations $\rho \in \Ih$ is abelian, since it is equivalent to $\modules{A/\Ih}$. In particular, it is closed under subquotients showing
\[
\langle V \rangle \subset \modules{A/\Ih}.
\]
Since $V$ is finite dimensional, the category is strongly finitary by the implication from (c) to (a) of Proposition~\ref{P:Deligne}. 
\end{rem}

\begin{lem} \label{L:DimFromGen}
We keep the setup of Remark \ref{R:relations}. We have
 \begin{align}\label{Eq:Mor}
\langle V \rangle \cong \modules{A/\Ih}.
\end{align}
If $e$ and $e'$ are idempotents in $A$, then

\begin{equation} \label{E:ineq} 
\dim_k eV\cdot\dim_ke' V\geq e (A/\Ih)e'\end{equation}
The induced map
\begin{align}
e A e' \xrightarrow{\varphi}  e\End_k(V) e' \cong \Hom_k(eV, e'V) 
\end{align} 
is surjective if and only if we have 
 \begin{align}\label{E:eq}
\dim_k eV\cdot\dim_ke' V= \dim_k e (A/\Ih)e'
\end{align}
\end{lem}
\begin{proof}
To see \eqref{Eq:Mor}, we note that by construction $V$ is a finite dimensional faithful module over the artinian algebra $A/\Ih$. It follows that there is $n\geq 1$ and an injective $A/\Ih$-module homomorphism
\[
A/\Ih\to V^n,
\]
see e.g. \cite[(3.30) Lemma]{curtis1981methods1}.
This shows that $\langle V \rangle \subset \modules{A/\Ih}$ contains a projective generator of  $\modules{A/\Ih}$. Hence we have the claimed equivalence.
Note that $A/\Ih$ is finite dimensional.

Multiplying the exact sequence 
\[
0 \to \Ih \hookrightarrow  A \xrightarrow{\varphi} \End_k(V) 
\] 
with the idempotens $e$ and $e'$ yields an exact sequence of subspaces
\[
0 \to e\Ih e' \hookrightarrow e Ae' \xrightarrow{\varphi}  e\End_k(V) e' \cong \Hom_k(eV, e'V) 
\]
and hence
\[ e(A/\Ih)e'\subset \Hom_k(eV,e'V).\]
This shows the inequality \eqref{E:ineq} and also that equality of dimensions holds iff $\varphi$ is surjective.
\end{proof}

We will later pick up on the following explicit example.

\begin{ex}\label{ex:dimension}
Let $Q$ be a finite quiver of the form 
 \begin{equation*}\begin{tikzpicture}[description/.style={fill=white,inner sep=2pt}]
   \matrix (m) [matrix of math nodes, row sep=2em,
                 column sep=3em, text height=1.5ex, text depth=0.25ex,
                 inner sep=2pt, nodes={inner xsep=0.3333em, inner
ysep=0.3333em}] at (0, 0)
    {   & 2    \\
       1&&  4 \\
       & 3  \\
    };
%arrows={Stealth[length=8pt,open,bend,sep]}
 \path[-{Classical TikZ Rightarrow[length=1.5mm]}] ($(m-2-1.east)+(0, 1mm)$)  edge node[xshift=-1.5mm, yshift=1.5mm, scale=0.85]{$a$}  ($(m-1-2.west)$) ;
 \path[-{Classical TikZ Rightarrow[length=1.5mm]}] ($(m-2-1.east) +(0, -1mm)$) edge  node[xshift=-1.5mm, yshift=-1.5mm, scale=0.85]{$c$} ($(m-3-2.west)$) ;

   \path[{Classical TikZ Rightarrow[length=1.5mm]}-] ($(m-2-3.west)+(0, +1mm)$)  edge node[xshift=1.5mm, yshift=1.5mm, scale=0.85]{$b$}  ($(m-1-2.east)$) ;
 \path[{Classical TikZ Rightarrow[length=1.5mm]}-] ($(m-2-3.west) +(0, -1mm)$) edge node[xshift=1.5mm, yshift=-1.5mm, scale=0.85]{$d$}  ($(m-3-2.east)$) ;

\draw[ -{Classical TikZ Rightarrow[length=1.5mm]}] ($(m-2-1.south) + (-1mm,0mm)$) .. controls +(2.5mm,-35mm) and
+(-2.5mm,-35mm) ..  node[ yshift=-2mm, scale=0.85]{$E_\tau$} ($(m-2-3.south) + (1mm,0mm)$);
\draw[ -{Classical TikZ Rightarrow[length=1.5mm]}] ($(m-2-1.south) + (1mm,0mm)$) .. controls +(2.5mm,-25mm) and
+(-2.5mm,-25mm) ..  node[ yshift=-3mm, scale=0.85]{$\vdots$} node[ yshift=2mm, scale=0.85]{$E_1$} ($(m-2-3.south) + (-1mm,0mm)$); 
  
 \end{tikzpicture}\end{equation*}
and let $A=kQ$ be its path algebra. It has dimension $10+\tau$ (corresponding to
$4$ vertices, $4+\tau$ edges, $2$ paths of length $2$). Let $e_i \in A$ be the idempotent at vertex $i$.
In particular,
\begin{align}\label{E:DimE4E1}
 \dim_k e_4 kQ e_1=2+\tau
 \end{align}
corresponding to the $(2+\tau)$ paths starting in $1$ and ending in $4$.

Let $V$ be the representation of $Q$ which attaches a $1$-dimensional vector space to each vertex and an isomorphism to the edges $a, b, c, d$ and arbitrary maps on the edges $E_i$. As before let
$\Ih=\ker(\varphi\colon kQ\to\End_k(V))$. By Lemma~\ref{L:DimFromGen}
\begin{equation}\label{eq:ex1.6} \dim_k e_4 (kQ/\Ih) e_1 \leq \dim_ke_4V\cdot \dim_k e_1V=1.
\end{equation}
The path $ba$ has non-zero-image under $\varphi$, hence $\varphi\colon e_4 kQ e_1 \to e_4 \End_k(V) e_1$ is surjective. This implies that we have even equality in
\eqref{eq:ex1.6}.
In combination with \eqref{E:DimE4E1} above, this implies $\dim_k e_4\Ih e_1 = 1+\tau$

Now assume that $\Ih \subset \rad(A)^2$. In other words, $\Ih \subset (ba, dc)= \rad(A)^2$ is at most two dimensional. 
 We claim that $\Ih=(ba-\lambda dc)$ for some $0 \neq \lambda \in k$. Indeed, by definition, $V$ satisfies a relation of the form $(ba-\lambda dc)$, so $\Ih\supseteq (ba-\lambda dc)$. If the inclusion was strict, we would have $\Ih = (ba, dc)= \rad(A)^2$ for dimension reasons, contradicting our assumption that all morphisms at edges $a, b, c, d$ are isomorphisms and hence cannot compose to zero.

Summing up, if $\Ih \subset
 \rad(A)^2$, we have $\Ih=(ba-\lambda dc)$ for some $0 \neq \lambda \in k$, $\tau=0$ and
\[ \langle V\rangle\isom \modules{kQ/(ba- \lambda dc)}.\]
\end{ex}

\subsection{Morita theory: from abelian categories to basic algebras}\label{ssec:morita}
The goal of this section is to show that for every finite dimensional algebra $A$ there is a finite dimensional \emph{basic} algebra $B$ with an equivalent module category. This will later be used to describe the relationship between the $k$-dimension of $A$ and $B$ explicitly.

We start by explaining what basic algebras are.

\begin{lem}\label{lem:idempotents}
Let $A$ be a finite dimensional $k$-algebra.
There is a complete set of primitive orthogonal idempotents $e_1, \dots,e_n\in A$ inducing an isomorphism of left $A$-modules
\begin{align}\label{E:decomp-into-indec-proj}
{}_AA \cong Ae_1 \oplus \cdots \oplus Ae_n.
\end{align}
In particular, each $Ae_i$ is indecomposable.
\end{lem}
\begin{proof}
See e.g. \cite[Prop. I.4.8]{ARS}.
\end{proof}
Note that  each $Ae_i$ is projective. It is the \emph{projective cover} of the
simple quotient $S_i=Ae_i/\rad(A) Ae_i$.

\begin{defn}
A finite dimensional algebra $A$ is called \emph{basic} if the indecomposable modules $Ae_i$ in \eqref{E:decomp-into-indec-proj} are pairwise non-isomorphic.
\end{defn}
Equivalently, the simple quotients $S_i$ are pairwise non-isomorphic. This is the case if and only if $A/\rad(A)$ is a product of skew fields.

The following well-known example yields `natural' \emph{non} basic algebras.
\begin{ex}
Let $G$ be a finite group and let $k$ be an algebraically closed field such that $\Char(k) \nmid |G|$. 
Then the group algebra $k[G]$ is basic if and only if $G$ is abelian.

Indeed, by Maschke's theorem, the group algebra $k[G]$ is semi-simple. Since $k$ is algebraically closed, this yields an isomorphism
\begin{align}\label{E:Wedd}
k[G] \cong M_{n_1\times n_1}(k) \times \ldots \times M_{n_s \times n_s}(k)
\end{align} 
for some $n_i>0$. Now, $G$ is abelian if and only if the group algebra is commutative, which by \eqref{E:Wedd} holds if and only if all $n_i=1$. But this is equivalent to $k[G]$ being basic.

\end{ex}

We fix some notation. By \eqref{E:decomp-into-indec-proj}  there is an isomorphism
\begin{align}\label{E:decomp}
{}_AA \cong \bigoplus_{i=1}^n P_i^{\oplus m_i}
\end{align}
where the $P_i \cong Ae_i$ are the pairwise non-isomorphic indecomposable left $A$-modules and the $m_i$ are positive integers.
\begin{defn}\label{defn:mult}
We call $m_i$ the \emph{multiplicity} of $P_i$ in $A$.
\end{defn}
Let $
P = \bigoplus_{i=1}^n P_i  \, \text{ and } \, B=\End_A(P)\op.
$
The next result follows from Morita theory. We follow the approach in \cite[Section II.2]{ARS}.

\begin{lem}[{\cite[Corollary II.2.6]{ARS}}]\label{lem:Morita}
Using the notation introduced above, there is an exact $k$-linear equivalence 
\[
F:=\Hom_A(P, -) \colon \modules{A} \to \modules{B}. 
\] 
\end{lem}

\begin{cor}\label{cor:basic} Keeping the notation above, we have
\begin{enumerate}
\item \label{it:1}The indecomposable projective $B$-modules are given by $Q_i=F(P_i)$.
\item $B$ is a basic algebra.
\item $A \cong \End_B(\bigoplus_{i=1}^n Q_i^{\oplus m_i})\op$
\end{enumerate}
\end{cor}
\begin{proof}
Since $F$ is an equivalence it sends indecomposable projective $A$-mo\-du\-les to indecomposable projective $B$-modules and every indecomposable projective $B$-module arises in this way. Now equation \eqref{E:decomp} implies assertion (\ref{it:1}). 
Since the $P_i$ are pairwise non-isomorphic by assumption and $F$ is an equivalence, the isomorphism 
\[
{}_BB \cong \Hom_A(P, P) \cong \bigoplus_{i=1}^n F(P_i)
\]
shows that $B$ is basic. The last assertion follows by taking opposite algebras in 
\[
A\op \cong \End_A({}_AA) \cong \End_B(F({}_AA)) \cong \End_B\left(\bigoplus_{i=1}^n Q_i^{\oplus m_i}\right).
\]
\end{proof}

\begin{ex}\label{ex:morita}
Let $A=M_{r \times r}(k)$ be the algebra of square matrices of size $r$, then the `space of column vectors' $M_{1 \times r}(k)$ is equal to the projective $P=P_1$ with multiplicity $m_1=r$. We have $B \cong k$.
\end{ex}

\subsection{Species and path algebras}\label{sec:species}
Corollary \ref{cor:basic} reduces the computation of the $k$-dimension of a finite dimensional algebra $A$ (with known `multiplicities' $m_i$) to understanding dimensions of Hom-spaces between indecomposable projective modules $Q_i$ over its basic algebra $B$. Our next aim is to review the structure theory of basic algebras in more detail.
The notion of a \emph{species} and its path algebra was introduced by
Gabriel.
\begin{defn}[{\cite{Gabriel2}, \cite[Section 2]{Berg}}]
Let $I$ be an index set. A \emph{$k$-species} $\cs=(D_i, _{j}E_i)_{i, j \in I}$ is a set of set of division rings  $D_i$ and $D_j-D_i$-bimodules $_{j}E_i$ such that
\begin{itemize}
\item The $D_i$ are finite dimensional $k$-algebras and $\lambda x = x \lambda$ for all $x \in D_i$ and all $\lambda \in k$.
\item The $_{j}E_i$ are finite dimensional
over $k$ and $\lambda x = x \lambda$ for all $x \in {_{j}E_i}$ and all $\lambda \in k$.
\end{itemize}

In the sequel, we often write \emph{species} for \emph{$k$-species} if the base field $k$ is clear from the context.

A species is called \emph{finite} if the index set $I$ is finite.
\end{defn}

Now we associate a $k$-algebra $k\cs$ with a species $\cs$.

\begin{defn} Let $\cs=(D_i, _{j}E_i)_{i, j \in  I}$ be a species. Set 
\begin{align}
D=\bigoplus_{i \in I} D_i \qquad E=\bigoplus_{i, j \in I} \ _{j}E_i 
\end{align}
Then $E$ is a $D\!-\!D$-bimodule and we can form the tensor ring
 \begin{align}
k\cs:= T_D(E)=D \oplus E \oplus (E \otimes_D E) \oplus  (E \otimes_D E \otimes_D E)  \oplus \cdots
\end{align}
Since $k$ acts centrally on $D$ and on $E$, the ring $k\cs$ is a $k$-algebra, which we call the \emph{path algebra} of $\cs$.

We denote by $\rr_{\cs}\subset k\cs$ the ideal generated by $E$.
\end{defn}

 By construction, the path algebra of a finite species is basic with set of simple objects $I$. It is hereditary  by \cite{ButlerKing}. For use in the non-basic case, we also consider the following generalisation (note that the algebras in \cite{ButlerKing} are also not assumed to be basic):

\begin{defn}\label{defn:species_with_mult}
Let $I$ be an index set. A \emph{$k$-species with multiplicities} $\cs=(K_i, _{j}E_i)_{i, j \in I}$ is a set of simple $k$-algebras $K_i$ and $K_j-K_i$-bimodules $_{j}E_i$ such that
\begin{itemize}
\item The $K_i$ are finite dimensional $k$-algebras and $\lambda x = x \lambda$ for all $x \in K_i$ and all $\lambda \in k$.
\item The $_{j}E_i$ are finite dimensional
over $k$ and $\lambda x = x \lambda$ for all $x \in {_{j}E_i}$ and all $\lambda \in k$.
\end{itemize}
The number $m_i$ such that $K_i\isom M_{m_i}(D_i)$ for a division algebra $D_i$ is called \emph{multiplicity} of $i$.
We set
\begin{align}
K=\bigoplus_{i \in I} K_i \qquad E=\bigoplus_{i, j \in I} \ _{j}E_i 
\end{align}
Then $E$ is a $K\!-\!K$-bimodule and we can form the tensor ring
 \begin{align}
k\cs:= T_K(E)=K \oplus E \oplus (E \otimes_K E) \oplus  (E \otimes_K E \otimes_K E)  \oplus \cdots
\end{align}
We call the \emph{path algebra} of $\cs$.

We denote by $\rr_{\cs}\subset k\cs$ the ideal generated by $E$.
\end{defn}

Unless otherwise specified, from now on all our species are assumed to be finite.
The path algebra is still hereditary by \cite{ButlerKing}.

\begin{defn}\label{defn:q_from_cs}
 Let $\cs$ be a species (with multiplicities). The quiver $Q_\cs$ of $\cs$ is given by
the set of vertices $I$ and an edge $e$ from $j$ to $i$ whenever
$_jE_i$ is non-zero. 

We say that $\cs$ has no oriented cycles, if $Q_\cs$  has no oriented cycles.

Given a path $\gamma=\varepsilon_1\cdots \varepsilon_n:v\pfad w$ in $Q_\cs$, we put
\[ 
E(\gamma)={}_vE_{t(\varepsilon_n)}\tensor_{K_{t(\varepsilon_n)}}{}_{s(\varepsilon_{n-1})}E_{t(\varepsilon_{n-1})}\tensor_{K_{t(\varepsilon_{n-1})}}\dots \tensor_{K_{t(\varepsilon_2)}} {}_{s(\varepsilon_1)}E_w
\]
for the tensor product of bi-modules along the path $\gamma$. If $\gamma:v\pfad v$ has length zero, we put $E(\gamma)=K_v$. Varying over all paths we put
\begin{equation}\label{Eq:Eij}
E(ij)= \bigoplus_{\gamma:i\pfad j}E(\gamma).
\end{equation}
\end{defn}
In this notation
\begin{equation}\label{eq:dim_kcs} k\cs=\bigoplus_{i,j\in V(Q_\cs)}E(ij).\end{equation}

\begin{rem}\label{R:NoCycles} If $\cs$ has no oriented cycles, then $k\cs$ is finite dimensional and $\rr_\cs=\rad(k\cs)$.
\end{rem}

\begin{ex} \label{ex:Q}
Let $Q$ be a finite quiver. It defines a species $\cs$ with $I=V(Q)$ the vertex set, $K_i=k$ for all vertices and $_jE_i= k^{{}_jd_i}$ where ${}_jd_i$ is the number of edges from $j$ to $i$ in $Q$. In particular, all multiplicities are $1$ in this case.
The quiver of 
$\cs$ has the same vertices as $Q$, but collapses multiple edges from $i$ to $j$ into a single edge. The path algebras agree:
\[ kQ=k\cs.\]
\end{ex}

\begin{defn}\label{defn:species_of_c}
Let $\Ch$ be a strongly finitary category, see Definition~\ref{defn:str_finitary}. It gives naturally rise to a species $\cs_\Ch$:
\begin{itemize}
\item Let  $I=\{1,\dots,n\}$ with $S_1,\dots,S_n$ the set of simple objects of $\Ch$,
\item for each simple object $S_i$ let $D_i=\End_{\Ch}(S_i)$,
\item for each pair $(i,j)$ let
\[ _{j}E_i=D\Ext^1_\Ch(S_j,S_i)\]
with the operations of $D_j$ and $D_i$ via the first and second argument, respectively.
\end{itemize}
In the case of $\Ch=\modules{A}$ we write $\cs_A$ for $\cs_{\modules{A}}$.
\end{defn}
This applies to our categories $\modules{A}$. In later chapters, we will also apply it to categories of motives.

\begin{rem}\label{rem:is_surj}
If $\Ch$ is strongly finitary and $\Ch'\subset\Ch$ is a full abelian subcategory closed under subquotients (hence again strongly finitary), then
the species and its path algebra are functorial. The map
\[ k\cs_\Ch\to k\cs_{\Ch'}\]
is surjective because $\Ext^1_{\Ch'}(S_j,S_i)\subset\Ext^1_{\Ch}(S_j,S_i)$ for all simple objects $S_i$ and $S_j$ of $\Ch'$.
\end{rem}

We also have a use for a variant with multiplicities.
\begin{defn}\label{defn:csmult_A}
Let $A$ be a finite dimensional $k$-algebra with species
$\cs_A=(D_i,{}_jE_i)$.  The \emph{species with multiplicities}
$\csmult_A$ has the same vertices as $\cs$, corresponding to the set of simple
$A$-modules $S_1,\dots,S_n$,
\begin{itemize}
\item for each simple $i$ let $K_i$ be the $S_i$-isotypical component of
$K=A/\rad(A)$,
\item for each pair $(i,j)$ let
\[ _{j}E_i=D\Ext^1_A(K_j,K_i)\]
with the operations of $K_j$ and $K_i$ via the first and second argument, respectively.
\end{itemize}
\end{defn}
Then $K_i\isom M_{m_i}(D_i)$ where $D_i=\End_A(S_i)$ and $m_i$ is the multiplicity in the sense of Definition~\ref{defn:mult}. We have
$K_i=S_i^{m_i}$ as an $A$-module and hence
\begin{equation}\label{eq:ext_mult}
 \Ext^1_A(K_j,K_i)\isom\Ext^1_A(S_j,S_i)^{m_jm_i}.\end{equation}

%\begin{prop}[{ \cite[Theorem~3.12]{Berg} }]\label{P:SpeciesWithRel} 
%Let $k$ be perfect and
%$B$ a basic finite dimensional $k$-algebra. Then there 
%exists
%a two-sided admissible ideal 
%$\ci \subset k\cs_B$ and a $k$-algebra isomorphism
%\begin{align}\label{E:QuotientOfHereditary}
%B \cong k\cs_B/\ci.
%\end{align}
%\end{prop}
%An ideal $\ci$ is called \emph{admissible} if it is contained in $\rr_{\cs_B}^2$ and there is $n\geq 2$ such that $\rr_{\cs_B}^n\subset \ci$.
%

The following result is formulated and proved for basic algebras in \cite[Theorem~3.12]{Berg}.
\begin{prop}\label{P:SpeciesWithRel} 
Let $k$ be perfect and
$A$ a  finite dimensional $k$-algebra. Then there 
exists
a two-sided admissible ideal 
$\ci \subset k\csmult_A$ and a $k$-algebra isomorphism
\begin{align}\label{E:QuotientOfHereditary}
A \cong k\csmult_A/\ci.
\end{align}
\end{prop}
An ideal $\ci$ is called \emph{admissible} if it is contained in $\rr_{\csmult_A}^2$ and there is $n\geq 2$ such that $\rr_{\csmult_A}^n\subset \ci$.

\begin{proof}
Recall that $K=\bigoplus_i K_i\isom A/\rad(A)$. By Lemma~\ref{lem:compute_ext},
\[ E=\bigoplus_{i,j}{}_jE_i\isom \rad(A)/\rad(A)^2.\]
We  now show the existence of a surjective map of $k$-algebras
\[ T_K E\to A.\]
We follow \cite{Berg} where the basic case is handled.

\textit{1. Step:} using the Theorem of Wedderburn and Malcev \cite[Theorem~11.6]{pierce},
we show that there exists a $K$-algebra splitting $K\to A$, cf.~\cite[Prop. 3.10]{Berg}.
To show that this theorem is applicable,
we need to verify the following statements (for $R=k$ and ${\bf J}(A)=\rad(A)$):
\begin{enumerate}
\item[(a)] $\dim A/\rad(A)\leq 1$ (here $\dim$ is the homological dimension of $A$ in Hochschild cohomology over $k$). 
\item[(b)] $A/\rad(A)$ is projective as a $k$-module
\item[(c)] $\rad(A)^n=0$ for some $n\geq 1$.
\end{enumerate}
The first holds because $A/\rad(A)$ is semi-simple and thus $A/\rad(A) \otimes_k (A/\rad(A))^{\op}$ is semi-simple since $k$ is perfect.
The second follows because $k$ is a field. The third holds because $A$ is Artinian.

We fix such a splitting $K \to A$ and use it to view the $A-A$-bimodules $\rad(A)^n$ as $K-K$-bimodules. Put $E=\rad(A)/\rad(A)^2$ and also view it as a $K-K$-bimodule.

\textit{2. Step:} We choose a $K-K$-linear splitting of
\[  \rad(A)\to E.\]
It exists because  $K\tensor_kK^{\op}$ is semi-simple (as seen above). %This is the place where we use that $k$ is perfect. 
(This is the argument given in \cite[Proposition~3.10]{Berg}.)

By the universal property of tensor algebras, we get an induced map $\pi\colon T_K(E)\to A$.

\textit{3. Step:} we show that $\pi$ is surjective 
(cf.~\cite[Proposition~3.2]{Berg}). 

We argue by induction over $N$ such that
$\rad(A)^N=0$. The statement holds for $N=1$. Assume it holds for $N$. We consider $N+1$. By inductive assumption, $\pi_N \colon T_K(E)\xrightarrow{\pi} A \to A/\rad(A)^N$ is surjective. 
Let $a\in A$ with image $\bar{a}$ in $A/\rad(A)^N$. Since $\pi_N$ is surjective, there is
$x\in T_K(E)$ with $\pi_N(x)=\bar{a}$. By construction, $a-\pi(x)\in\rad(A)^N$. The map
\[ \rad(A)^{\tensor N}\to \rad(A)^N\]
is surjective. It factors via $E^{\tensor N}$  because $\rad(A)^{N+1}=0$. Hence,
$a-\pi(x)$ is in the image of $\pi$ and thus $a$ is in the image as well.

By construction, the kernel of $\pi$ is contained in $\rr_{\csmult_A}^2$.
\end{proof}

\begin{rem}The notion of a species with multiplicities is well-behaved under base change. Let $l/k$ be an extension of perfect fields. For all
finite dimensional $k$-algebras $A$, we have a canonical isomorphism
\[ (k\csmult_A)_l\isom l\csmult_{A_l}\]
compatible with the presentation of $A_l$ in Proposition~\ref{P:SpeciesWithRel}.

Note that even if $B$ is a basic $k$-algebra, its base change $B_l$ is not necessarily basic. (Consider for example $B=D$ a finite dimensional division algebra.) 

The quiver of the species with multiplicities of $A_l$ need \emph{not} be the same as the one of $A$. (Consider for example $B=\Q(i)$ over $k=\Q$ and $l=\Qbar$.)
\end{rem}

If $B$ is basic, the hereditary algebra $k\cs_B$ in \eqref{E:QuotientOfHereditary} corresponds to the `saturation' of $\modules{B}$ in Section~\ref{sec3}.

%% file: kapitel2b.tex
\section{Relations and higher extensions}\label{sec:ext}
Throughout the chapter we consider a finite dimensional $k$-algebra $A$ that can be represented in the form
\[ A=H/\ci\]
with a hereditary algebra $H$ and a two-sided ideal $\ci\subset H$.
Let $K=A/\rad(A)$ and $\rr_H$ the kernel of $H\to A\to K$. 

\begin{rem}
We do not assume that $H$ is finite dimensional. The ideal $\rr_H$ takes the role of the radical. However, our assumptions do not imply that $\rr_H=\rad(H)$, even if $H$ is finite dimensional. 
%The main case of interest for us is the case where $A$ is basic and $H$ is the hereditary algebra of 
%Proposition~\ref{P:SpeciesWithRel}. 
If $H$ is finite dimensional and 
 $\ci$ admissible (i.e., $\ci\subset\rad(H)^2$), then $\rr_H=\rad(H)$, cf.~Remark \ref{R:NoCycles}.
\end{rem}

The situation can be summed up in the commutative diagram
\[\begin{xy}\xymatrix{
		 &\rr_H\ar@{^{(}->}[d]\ar@{->>}[r]&\rad(A)\ar@{^{(}->}[d]\\
\ci\ar@{^{(}->}[ru]\ar@{^{(}->}[r]&H\ar@{->>}[r]\ar@{->>}[rd]&A\ar@{->>}[d]\\
&&K
}\end{xy}\]

The aim of this section is to describe the interaction between the ideals $\ci$ and $\rr_H$ in terms of extension groups. This will allow us to deduce dimension formulas in Section~\ref{sec:dim}. 
\subsection{Formulas for Ext and Tor}

The $\Tor$-formulas in the following result can also be deduced\footnote{We thank Julian K{\"u}lshammer for pointing out this reference.} from \cite{ButlerKing}, see also the earlier work of Bongartz \cite[Theorem 1.1.]{Bongartz} over algebraically closed fields. For the convenience of the reader, we include a proof. We use the notation
$D(-)=\Hom_k(-,k)$.

\begin{thm}\label{T:Bongartz}
Let $k$ be perfect, $A$ a finite dimensional $k$-algebra with $K=A/\rad(A)$.
Let $H$ be a hereditary $k$-algebra (not necessarily finite dimensional) and
$H\to A$ a surjective $k$-algebra homomorphism. Let $\rr_H$ be the kernel of $H\to A \to K$ and $\ci$ the kernel of $H\to A$. 
	Then there are isomorphisms of $K-K$-bimodules  
\begin{align}
D\Ext^{2n}_A(K,K)&\isom \Tor_{2n}^A(K, K) \\
 &\isom (\ci^n \cap \mathfrak{r}_H\ci^{n-1}\mathfrak{r}_H)/(\ci^n \mathfrak{r}_H +\mathfrak{r}_H\ci^n) \text{, for } n \geq 1.  \\
D\Ext^{2n+1}_A(K,K)&\isom \Tor_{2n+1}^A(K, K) \\
 &\isom (\mathfrak{r}_H\ci^n \cap \ci^{n}\mathfrak{r}_H)/(\ci^{n+1} +\mathfrak{r}_H\ci^n\mathfrak{r}_H) \text{, for } n \geq 0. \label{E:Bong2}
\end{align}
\end{thm}
\begin{rem}
\begin{enumerate}
\item By Proposition~\ref{P:SpeciesWithRel}, $H=k\csmult_A$ satisfies the assumptions of Theorem \ref{T:Bongartz}. 
This is the case we will be interested in.
\item
If $k$ is not perfect, the formula for $\Tor^A$ still holds true. 
\end{enumerate}
\end{rem}

We divide the proof into several lemmas. 

%\annette{obsolet}
%\begin{lem}\label{lem:Ext_Tor}
%Let $A$ and $K$ be as above.
%For all $K$-modules $X$, the trace $\tr:K\to k$ induces a natural isomorphism
%\[ DX\isom \Hom_K(X,K)\]
%of $K\op$-modules. Moreover,
%$\Tor_{m}^A(K, K) \cong D\Ext^{m}_A(K, K)$ as $K-K$-bimodules.
%\end{lem}
%\begin{proof}
% The first claim is an instance of \cite[Theorem 3.1]{Rickard}. Indeed, 
%finite dimensional skew fields over perfect fields are symmetric via the trace map,
%see \cite[VIII, p.~375, Cor.]{bourbaki8}. The property extends to matrix algebras over such skew fields and by Wedderburn's Theorem to $K$.
%
%In particular, 
% \[ D(K_K)\isom {}_KK,\quad D({}_KK)\isom K_K\]
%as $K$-modules and $K\op$-modules, respectively. 
%
%
%We now claim that 
%\begin{equation}\label{ext-tor} DY\otimes_AX\cong D\Hom_A(X,Y)\end{equation}
%for all finitely generated $A$-modules. There is a natural map from the left
%to the right:
%\[ f\otimes x\mapsto ( g\mapsto f(g(x)).\]
%Both sides are right exact in $X$ as $D$ is exact. By taking a free resolution
%of $X$, we are reduced to the case $X=A$. In this case the map is the 
%\[ \mathrm{id}: DY\to DY.\]
%The isomorphism \eqref{ext-tor} passes to derived functors, so that we have
%\[ \Tor_n^A(DY,X)\isom D\Ext^n_A(X,Y).\]
%The isomorphism is natural in both arguments.
%We evaluate in $X=Y={}_KK$. The claim
%follows because $DK\cong K\op$. It is $K-K$-equivariant by naturality.
%\end{proof}

The following lemma is well-known, cf. e.g. \cite[Proposition 1]{ENN}.

\begin{lem}\label{L:Proj}
Let $A$ be a $k$-algebra and $I \subset A$ a two-sided ideal.
Assume $A=P \oplus P'$ as left $A$-modules. Then $P/IP$ is a projective $A/I$-module.
\end{lem}
\begin{proof}
Indeed, $A/I \otimes_A P$ is a direct summand of $A/I \otimes_A A \cong A/I$ and therefore projective
as an $A/I$-module. The isomorphism  $A/I \otimes_A P \to P/IP$ completes the proof.
\end{proof}

The following observation is the key part in the proof of Theorem \ref{T:Bongartz}, cf. \cite[Proof of Proposition 3]{ENN}.
\begin{lem}\label{L:ProjR}
Let $A=H/\ci$ be a finite dimensional  algebra as in Theorem \ref{T:Bongartz} and let $K=A/\rad(A)$. Then there is a projective resolution of $A$-modules
\begin{align}\label{E:resol}
\cdots \to P_6 \to P_5 \to \ci^2/\ci^3 \to \ci\mathfrak{r}_H/\ci^2\mathfrak{r}_H \to \ci/\ci^2 \to \mathfrak{r}_H/\ci\mathfrak{r}_H \to A \to K \to 0,
\end{align}
where $P_{2l}:=\ci^l/\ci^{l+1}$ and $P_{2l+1}:=\ci^l\mathfrak{r}_H/\ci^{l+1}\mathfrak{r}_H$ for all $l \geq 0$ and the maps $\partial_m\colon P_m \to P_{m-1}$ are induced by the inclusions of $H-H$-bimodules
\begin{align} \label{E:incl}
\cdots  \ci^3 \mathfrak{r}_H  \subset \ci^3 \subset \ci^2 \mathfrak{r}_H  \subset \ci^2 \subset  \ci \mathfrak{r}_H  \subset \ci \subset \mathfrak{r}_H \subset \ci^0=H.
\end{align}
The complex \eqref{E:resol} is a complex of $A-H$-bimodules.
\end{lem}
\begin{proof}
Since $H$ is hereditary every left ideal $P \subset H$ is a projective $H$-module. Now Lemma \ref{L:Proj} shows that the $P_i$ are projective $A=H/\ci$-modules. The inclusions in \eqref{E:incl} yield short exact sequences of $A$-modules $A-H$-bimodules
\begin{gather}
0 \to \ci^l \mathfrak{r}_H/\ci^{l+1} \to \ci^l/\ci^{l+1} \to \ci^l/\ci^l \mathfrak{r}_H \to 0 \label{E:K_2l+1} \\
0 \to  \ci^{l+1}/\ci^{l+1} \mathfrak{r}_H \to \ci^l\mathfrak{r}_H/\ci^{l+1}\mathfrak{r}_H \to \ci^l\mathfrak{r}_H/\ci^{l+1} \to 0 
\end{gather}
for all $l \geq 0$ showing that the sequence \eqref{E:resol} is exact -- note that $H/\mathfrak{r}_H \cong K=A/\rad(A) $ as $A-H$-bimodules. This completes the proof.
\end{proof}

\begin{rem}
In general, the projective resolution in Lemma \ref{L:ProjR} is far from being a minimal projective resolution. Indeed, consider, for example, the following quiver
\begin{equation*}\begin{tikzpicture}[description/.style={fill=white,inner sep=2pt}]
   \matrix (m) [matrix of math nodes, row sep=3em,
                 column sep=2.5em, text height=1.5ex, text depth=0.25ex,
                 inner sep=2pt, nodes={inner xsep=0.3333em, inner
ysep=0.3333em}] at (0, 0)
    {  1&2&3& \cdots & n-1 &n. \\};
 \path[->] ($(m-1-1.east)$)   edge  node [ scale=0.75, yshift=2.5mm] [midway] {$a_1$}  ($(m-1-2.west)$) ;
 \path[->] ($(m-1-2.east)$)  edge  node  [ scale=0.75, yshift=2.5mm] [midway] {$a_2$}  ($(m-1-3.west)$) ;
 \path[->] ($(m-1-3.east)$)  edge node  [ scale=0.75, yshift=2.5mm] [midway] {$a_{3}$}  ($(m-1-4.west)$) ;
\path[->] ($(m-1-4.east)$)  edge node  [ scale=0.75, yshift=2.5mm] [midway] {$a_{n-2}$}  ($(m-1-5.west)$) ;
 \path[->] ($(m-1-5.east)$)  edge  node  [ scale=0.75, yshift=2.5mm][midway] {$a_{n-1}$} ($(m-1-6.west)$) ;
 \end{tikzpicture}\end{equation*}
 Let $I \subset kQ$ be the two-sided ideal generated by the following paths of length~$2$
 \[
 a_3a_2, \ a_5a_4, \ a_7a_6, \ \ldots, \ a_{n-2}a_{n-3}.
 \]
Let $A=kQ/I$. One can compute that $\gldim A=2$. If $n$ is odd, then 
\[
0 \neq a_{n-1}a_{n-2}a_{n-3}\cdot \ldots \cdot a_3a_2 \in I^{\frac{n-3}{2}}\mathfrak{r}_{kQ} 
\]
shows that the projective resolution of $K$ in Lemma \ref{L:ProjR} has length $n-2$. 
\end{rem}

\begin{lem}\label{L:CE} Let $M$ be an $A$-module. Then, in the notation of Lemma \ref{L:ProjR},
\[
\Tor_A^m(M, K)=\ker(M\otimes_A \ker \partial_{m-1}  \to M \otimes_A P_{m-1})
\]
\end{lem}
\begin{proof}
Since $M \otimes_A -$ is right exact, the statement follows for example from \cite[Exercise V.10.1 on p. 104]{CE}.
\end{proof} 

In the computations below, we will repeatedly use the following isomorphism (using the notation of Lemma \ref{L:ProjR}) of $A-H$-bimodules for all $A-H$-bimodules $M$
\begin{align}\label{E:modrad}
K \otimes_A M \cong M/\mathfrak{r}_H M.
\end{align}

\begin{lem}\label{L:KerTimesK}
\begin{align}
K\otimes_A \ker \partial_{2n-1} \cong \ci^n/(\mathfrak{r}_H\ci^n\ + \ci^{n} \mathfrak{r}_H) \text{ for all }n \geq 1 \\
K\otimes_A \ker \partial_{2n} \cong \ci^n \mathfrak{r}_H/(\mathfrak{r}_H\ci^n\mathfrak{r}_H + \ci^{n+1}) \text{ for all }n \geq 0 
\end{align}

\end{lem}
\begin{proof}
We only do the computation for the second case. The first case is similar. By \eqref{E:K_2l+1}, we have
\[
\ker \partial_{2n} \cong \ci^n\mathfrak{r}_H/\ci^{n+1}
\]
Now, 
\[
\mathfrak{r}_H(\ci^n\mathfrak{r}_H/\ci^{n+1}) \cong (\mathfrak{r}_H\ci^n\mathfrak{r}_H + \ci^{n+1})/\ci^{n+1}
\]

Thus, applying \eqref{E:modrad} to $M=\ci^n\mathfrak{r}_H/\ci^{n+1}$   yields
\[
K\otimes_A \ker \partial_{2n} \cong \frac{\ci^n\mathfrak{r}_H/\ci^{n+1}}{(\mathfrak{r}_H\ci^n\mathfrak{r}_H + \ci^{n+1})/\ci^{n+1}}
\cong \frac{\ci^n \mathfrak{r}_H}{(\mathfrak{r}_H\ci^n\mathfrak{r}_H + \ci^{n+1})}
\]
using the isomorphism theorem in the last step.
\end{proof}

\begin{lem}\label{L:ProjTimesK}
\begin{align}
K\otimes_A P_{2n-1} \cong \ci^{n-1}\mathfrak{r}_H/(\mathfrak{r}_H\ci^{n-1}\mathfrak{r}_H + \ci^{n} \mathfrak{r}_H) \text{ for all }n \geq 1 \\
K\otimes_A P_{2n} \cong \ci^n/(\mathfrak{r}_H\ci^n + \ci^{n+1}) \text{ for all }n \geq 0 
\end{align}
\end{lem}
\begin{proof}
Again, we only do the computation for the second case. By the construction in Lemma \ref{L:ProjR}, we have
\[
P_{2n}=\ci^n/\ci^{n+1}
\]
Now,
\[
\mathfrak{r}_H(\ci^n/\ci^{n+1}) \cong (\mathfrak{r}_H\ci^n + \ci^{n+1})/\ci^{n+1}.
\]
Thus, applying \eqref{E:modrad} to $M=\ci^n/\ci^{n+1}$ yields
\[
K\otimes_A P_{2n} \cong \frac{\ci^n/\ci^{n+1}}{(\mathfrak{r}_H\ci^n + \ci^{n+1})/\ci^{n+1}}
\cong \frac{\ci^n}{(\mathfrak{r}_H\ci^n + \ci^{n+1})}
\]
\end{proof}

\begin{proof}[Proof of Theorem \ref{T:Bongartz}]
We only prove \eqref{E:Bong2}. By Lemma~\ref{lem:Ext_Tor} it suffices to compute the Tor-group.
By Lemma~\ref{L:CE}, 
\[
\Tor^A_{2n+1}(K, K) \cong \ker(K\otimes_A \ker \partial_{2n}  \to K \otimes_A P_{2n})
\]
which by Lemmas \ref{L:KerTimesK} and \ref{L:ProjTimesK} is isomorphic to 
\begin{multline*}
\ker\left(\frac{\ci^n \mathfrak{r}_H}{(\mathfrak{r}_H\ci^n\mathfrak{r}_H + \ci^{n+1})}  \to \frac{\ci^n}{(\mathfrak{r}_H\ci^n + \ci^{n+1})}\right)  \\[1ex]
\cong \frac{(\mathfrak{r}_H\ci^n + \ci^{n+1}) \cap \ci^n \mathfrak{r}_H}{(\mathfrak{r}_H\ci^n\mathfrak{r}_H + \ci^{n+1})} \cong \frac{\mathfrak{r}_H\ci^n \cap \ci^n \mathfrak{r}_H}{(\mathfrak{r}_H\ci^n\mathfrak{r}_H + \ci^{n+1})}
\end{multline*}
where the last step uses that $\ci^{n+1}$ is contained in the ideal we quotient out by. All isomorphisms are isomorphisms of $A-H$-bimodules by construction. As all terms are $K-K$-bimodules, they are even isomorphisms of $K-K$-bimodules.
\end{proof}

\subsection{A formula for the relation space}

\begin{thm} \label{thm:generate_ideal}
Let $H$ be a finite dimensional hereditary algebra, $\ci\subset\rad(H)^2$ an admissible ideal, $A=H/\ci$.

Then
\begin{align}\label{E:Ext2}
	D\Ext^{2}_A(K, K)\cong& \ci /(\ci \rad(H) +\rad(H)\ci)
\end{align}
as $K-K$-bimodules.
Let $s\colon D\Ext^2_A(K,K)\to \ci$ be a section of (\ref{E:Ext2}). 
Then $\ci$ is generated by the image of $s$ as a two-sided ideal.
\end{thm}
\begin{proof} We are in the case $\rr_H=\rad(H)$.
The formula is the case $n=1$ of Theorem~\ref{T:Bongartz} because we assume $\ci\subset\rr_H^2$. It implies
\[ \ci= \im(s)+ \ci\rr_H+\rr_H\ci.\]
Recursively, this implies
\[ \ci=\im(s)+\im(s)\rr_H+\rr_H\im(s)+\ci\rr_H^2+\rr_H\ci\rr_H+\rr_H^2\ci\]
etc. Note that the error terms $\rr_H^a\ci\rr_H^b$ are contained in $\rr_H^{a+b+2}$.  They vanish for $a,b$ big enough since the radical is nilpotent. 
\end{proof}

\begin{rem}
The second part of the Theorem is not true in general if $H$ is infinite dimensional, cf. e.g. \cite[Section 7]{BIKR}. There are examples such that 
$H/(\im(s))$ is infinite dimensional. This problem can be fixed. We restrict to the basic case,  $H\cong k\mathcal{S}$ for a species $\cs$. After passing to the $\rr_H$-adic completion $\hat{H}$ of $H$, the result still holds, see e.g. \cite[Section 3, in particular, Prop. 3.4.]{BIRSm}, where this is discussed in the case of quivers.
Note that if $H$ is finite dimensional then $\hat{H}=H$. For path algebras of quivers our statement (and its proof) are thus special cases of \cite[Prop. 3.4.]{BIRSm}.  
\end{rem}

\begin{cor}\label{C:Hereditary}
Assume that $H$ is finite dimensional and hereditary, $\ci\subset H$ an admissible ideal. The following conditions are equivalent
\begin{enumerate}
\item \label{I:hered}the algebra $A=H/\ci$ is hereditary, i.e. $\gldim A \leq 1$.
\item\label{I:Ext2=0} $\Ext_A^2(K, K)=0$.
\item\label{I:I=0} $\ci=0$.
\end{enumerate}
\end{cor}
\begin{proof}
The implications from (\refeq{I:hered}) to (\refeq{I:Ext2=0}) and from
(\refeq{I:I=0}) to (\refeq{I:hered}) are trivial. Assume (\ref{I:Ext2=0}).
	By Theorem~\ref{thm:generate_ideal} the ideal $\ci$ is generated by $D\Ext^2_A(K,K)=0$. This implies that $\ci=0$.
\end{proof}

The isomorphism \eqref{E:Ext2} has the following consequence, which we will use later.
\begin{cor}\label{C:LongRelations}
In the notation of Theorem~\ref{thm:generate_ideal}, assume that
\begin{align}\label{E:LongRelations}
	\ci \rad(H) =0=\rad(H) \ci.
\end{align}
Then 
\begin{align}
\ci \isom D\Ext^2_A(K, K).
\end{align}
	This happens in particular, if $\rad(H)^3=0$.
\end{cor}

\begin{thm}Let $A$ be a finite dimensional hereditary algebra over a perfect field $k$.
Then 
\[ A\isom T_KE\isom k\csmult_A\]
is isomorphic to a tensor algebra. 
\end{thm}
\begin{proof}We use Proposition~\ref{P:SpeciesWithRel} to write
$A$ a a quotient of $H=k\csmult_A$ by an admissible ideal $\ci$, i.e., $\ci\subset\rr_H^2$. As
$A$ is hereditary, the $n=1$ case of Theorem~\ref{T:Bongartz} gives
\[ \ci=\ci\rr_H+\rr_H\ci.\]
This implies $\ci\subset\rr_H^3$. Iteratively, we see that
$\ci\subset\rr_H^n$ for all $n\geq 1$. However,
\[ \bigcap_{n\geq 1}\rr_H^n=0\]
in the graded ring $T_KE$. Hence $\ci=0$.
\end{proof}

\begin{cor} A finite dimensional algebra over a perfect field is hereditary if and only if it is the path algebra of a species with multiplicities whose quiver is finite and does not contain any cycles.
\end{cor}
\begin{proof}All path algebras are hereditary. They are finite dimensional if and only if the quiver does not contain any cycles.
\end{proof}

\begin{rem}
The corollary is false for infinite dimensional hereditary algebras like formal power series rings $k\llbracket x \rrbracket$ over a field $k$.

Moreover, if $k$ is not perfect, there are finite dimensional hereditary algebras that are not tensor algebras -- they are quotients of a tensor algebra by a non-admissible ideal. This phenomenon seems to be first observed by Dlab \& Ringel, cf. e.g.~\cite[Section 5]{Berg}. 
\end{rem}

\subsection{Higher global dimension}

The following `generalises' Corollary~\ref{C:Hereditary}.
\begin{cor} Let $H$ be a finite dimensional hereditary algebra and $\ci\subset\rad(H)^2$ an admissible ideal.
	If $\rad(H)^{n+1} =0$, then $\gldim A \leq n$.
\end{cor}
\begin{proof} We have $\rad(H)=\rr_H$.
It suffices to show that $\Ext^{j}_A(K, K)=0$ for all $j>n$. Using $\ci \subset \mathfrak{r}_H^2$ and $\mathfrak{r}_H^{n+1} =0$, shows 
\begin{align*}
\ci^l \cap \mathfrak{r}_H\ci^{l-1}\mathfrak{r}_H &\subset \mathfrak{r}^{2l}_H  = 0, \quad \text{ for } l \geq \frac{n+1}{2} \\
\mathfrak{r}_H\ci^l \cap \ci^{l}\mathfrak{r}_H &\subset \mathfrak{r}^{2l+1}_H  = 0, \quad \text{ for } l \geq \frac{n}{2}.
\end{align*}
Now the formulas in Theorem~\ref{T:Bongartz}
complete the proof.
\end{proof}

%% file: kapitel2c.tex
\section{Dimension formulas for algebras}\label{sec:dim}
In this section, we translate the structural results into dimension formulas.

\subsection{Notation}\label{ssec:not_formulas}
We collect the notation introduced earlier for the convenience of the reader.
Throughout, $k$ is a perfect field and $A$ is a finite dimensional $k$-algebra.
\begin{itemize}
\item Let $S_1,\dots,S_n$ be the simple objects of $\modules{A}$;
\item  For $i=1,\dots, n$ let $m_i$ be the multiplicity of the projective cover $P_i$ of $S_i$ in the direct sum decomposition of $_{A}A$, see 
Section~\ref{ssec:morita}, Equation \eqref{E:decomp};
\item let $D_i=\End_A(S_i)$,
\item let $\Sh_A=(D_i, D\Ext^1_A(S_j,S_i))_{i,j\in\{1,\dots,n\}}$ be the species of $\modules{A}$, see Definition~\ref{defn:species_of_c} and let $\csmult_A=(\End_A(S_i^{m_i}), D\Ext^1_A(S_j^{m_j},S_i^{m_i})$ be the species with multiplicities of $A$, cf.~Definition~\ref{defn:csmult_A};
\item $Q_A$ the quiver  of the species, see Definition~\ref{defn:q_from_cs}, so it has an edge from $j$ to $i$ if and only if $D\Ext^1_A(S_j,S_i)$ is non-zero.
\item for each path $\gamma:v\pfad w$ we have $E(\gamma)$ of Definition~\ref{defn:q_from_cs};
\item for each pair $i,j$ we have $E(ij)=\bigoplus_{\gamma:i\pfad j}E(\gamma)$.
\end{itemize}

\begin{defn}\label{defn:delta}
We put
\begin{align*}
\delta(A)&=\dim_k A,\\
\delta_A(ij)&=\dim_k\Hom_A(P_i,P_j),\\
\delta_A(\gamma)&=\dim_k E(\gamma).\\
\end{align*}
\end{defn}

\begin{rem}\label{R:IndofMoritaandrelations}
The definitions of $\delta_A(ij)$ and $\delta_A(\gamma)$ do not take the multiplicities into account. In particular, they coincide for $A$ and the associated basic algebra $B$. Moreover, $\delta_A(\gamma)=\delta_{k\Sh_A}(\gamma)$ does not see the relations in $A$. 
\end{rem}

\subsection{The hereditary case}
The dimension formulas are easier in the hereditary case. This also implies upper bounds in general.
\begin{prop}\label{L:dimPathAlg}
Let $k$ be perfect, $A$ a finite dimensional $k$-algebra. Then
\begin{subequations}
\begin{align} 
\delta(A)&=\sum_{i,j=1}^n m_im_j\delta_A(ij),\\ 
\delta_A(ij)&\leq \sum_{\gamma:i\pfad j} \delta_A(\gamma), \label{eq:ij} \\ 
\delta_A(\gamma)&=\begin{cases}
 \dim_k D_v&\gamma:v\pfad v\text{ of length $0$}\\
\frac{\prod\limits_{\text{$\varepsilon:v\to w$ on $\gamma$}}\dim_k \Ext^1_A(S_v,S_w)}{\prod\limits_{v\in\gammacirc} \dim_k D_v} &|\gamma|\geq 1\label{eq:gamma}
\end{cases}
\end{align}
\end{subequations}
(where $\gammacirc$ is the set of vertices on the path which are different from the end points)
with equality if and only if $A$ is hereditary.
\end{prop}

\begin{proof} 
%Let $B$ be the basic algebra Morita equivalent to $A$, $F\colon\modules{A}\to\modules{B}$ the equivalence. Both $A$ and $B$ decompose as direct sums of the $P_i$ for $i=1,\dots,n$.
%By Corollary~\ref{cor:basic}, 
%\[ A\isom\End_B\left(\bigoplus_{i=1}^nQ_i^{\oplus m_i}\right)\op\]
%with $Q_i=F(P_i)$. Note that $\Hom_A(P_i,P_j)\isom \Hom_B(Q_i,Q_j)$ and hence $\delta_A(ij)=\delta_B(ij)$.  
%The formula for $\delta(A)$ follows by additivity. 

As in \eqref{E:decomp},
\[ {}_AA \cong \bigoplus_{i=1}^n P_i^{\oplus m_i}\]
and  
\[ A\isom\End_A\left(\bigoplus_{i=1}^nP_i^{\oplus m_i}\right)\op.\]
The formula for $\delta(A)$ follows by additivity.

As in Proposition~\ref{P:SpeciesWithRel}, 
there is hereditary algebra $H$ and an ideal $\ci\subset\rr_H^2$ such that $A\isom H/\ci$; in fact {$H=k\csmult_A$}. We may view $\modules{A}$ as a full subcategory of
$\modules{H}$. Let $P_i^H$ be the projective cover of $S_i$ in $\modules{H}$. We fix a surjection $\pi_i:P_i^H\to P_i$. As $P_i^H$ is projective, the dotted arrow in the diagram
\[\xymatrix{
  P_i^H\ar@{.>}[r]\ar[d] &P_j^H\ar[d]\\
  P_i \ar[r]^f &P_j
}\]
can always be filled in. This implies
\begin{align}\label{E:IneqDelta}
 \delta_A(ij)=\dim_k\Hom_A(P_i,P_j)\leq \dim_k\Hom_H(P_i^H,P_j^H)=\delta_H(ij).
\end{align} 
Note also that having equalities in \eqref{E:IneqDelta} for all $i,j$ implies that $\dim_kH=\dim_kA$, so that $A=H$ is hereditary.

To show 
\begin{align}\label{E:delta_H(ij)}
\delta_H(ij)=\sum_{\gamma:i\pfad j}  \delta_A(\gamma)
\end{align} 
 and thus finishing the proof of \eqref{eq:ij}, it suffices to assume that $A=k\Sh_A$ is basic and hereditary, cf. Remark \ref{R:IndofMoritaandrelations}.
By the description of $P_i$ inside of $A$ in Section~\ref{ssec:morita}
\[ P_i=A e_i =\bigoplus_{w, \gamma:i\pfad w}E(\gamma).\]
Comparison of \eqref{Eq:Eij} with Section~\ref{ssec:morita}
 gives
\[ \Hom_A(P_i,P_j)= \Hom_A(k\Sh_A e_i, k\Sh_A e_j)= e_j k\Sh_A e_i=  E(ij).\] 
This implies \eqref{E:delta_H(ij)}. 

By Remark \ref{R:IndofMoritaandrelations}, Formula \eqref{eq:gamma} follows from the definition of $E(\gamma)$ and
\[ \dim_k M\tensor_D N=\frac{\dim_k M\cdot \dim_k N}{\dim_k D}\]
for a skew field $D$ over $k$, a $D$-module $N$ and a $D\op$-module $M$. 
\end{proof}

There is a completely different, much easier upper bound.
\begin{lemma}\label{L:trivial_bound}
Let $V\in\modules{A}$ such that $\langle V\rangle=\modules{A}$. Then
\[ \delta(A)\leq (\dim_k V)^2.\]
Let $\tilde{e}_i,\tilde{e}_j\in A$ be the idempotents projecting to
$P_i^{m_i}$ and $P_j^{m_j}$, respectively. Then
\[ m_im_j\delta_A(ij)\leq \dim_k \tilde{e}_i V\dim_k\tilde{e}_jV.\]
\end{lemma}
\begin{proof}Apply Lemma~\ref{L:DimFromGen} to the algebra $A$. Note that the
$\Ih=\ker(A\to \End_k(V))$ vanishes in this case. The bound for
$\delta(A)$ follows with the idempotent $1$ in Lemma~\ref{L:DimFromGen}.

We have
\[ m_im_j\delta_A(ij)=\dim_k\Hom_A(P_i^{m_i},P_j^{m_j})=\dim_k\tilde{e}_iA\tilde{e}_j.\]
The upper bound is the one given in Lemma~\ref{L:DimFromGen}.
\end{proof}

\subsection{Refined dimension formulas}\label{ssec:refined}
We keep the general set-up and notation, but make assumptions on the global dimension.

\begin{cor}\label{c: lengthatmost2}Let $A$ be a finite dimensional algebra such that the maximal length of paths in $Q_A$ is at most $2$. Then
\[ \delta_A(ij)=\delta_{k\Sh_A}(ij)-m_im_j\dim_k\Ext^2_A(S_j,S_i)\]
and
\begin{align*} \delta(A)&=\sum_i m_i^2\dim_kD_i+\\
                       &+\sum_{i,j}m_im_j\dim_k\Ext^1_A(S_j,S_i)-\sum_{i,j}m_im_j\dim_k\Ext^2_A(S_j,S_i)\\
                       &+\sum_{a,b,c}m_a m_c\frac{\dim_k\Ext^1_A(S_b,S_a)\dim_k\Ext^1_A(S_c,S_b)}{\dim_k D_b}
\end{align*}
\end{cor}
Note that many terms in these sums vanish.
\begin{proof} The assumption implies $\rr^3=0$ where $\rr$ is the radical of
$k\Sh_A$. By Corollary~\ref{C:LongRelations}  
the basic algebra $B$ corresponding to $A$ takes the form
\[ B\isom k\Sh_A/\ci\]
where $\ci\isom D\Ext^2_B(K,K)\isom\bigoplus_{i,j}D\Ext^2_B(S_j,S_i)$.  

Combining the description of $\delta_A(ij)$ with the formulas in Proposition~\ref{L:dimPathAlg} for $k\Sh_A$, we get the formula for $\delta(A)$.
\end{proof}

The following result generalises the above formula to higher cohomological dimension. However, this estimate can become arbitrarily bad, cf. Example~\ref{Ex:dimEstimateLongerPaths} below.

\begin{cor}\label{C:IE} Let $A$ be a finite dimensional algebra, $K=A/\rad(A)$, $H=k\csmult_A$. For all $n \geq 2$, we have the following estimate
\begin{align}
\dim_k A \leq \dim_k H - \dim_k \Ext^2_A(K, K) - \ldots - \dim_k \Ext^n_A(K, K)
\end{align}
\end{cor}
\begin{proof}
Since $A$ is finite dimensional, the $\Ext^i_A(K, K)$ are finite dimensional. We can assume that $\dim_k H$ is finite, for otherwise the inequality holds trivially. As before $A\isom k\csmult_A/\ci$. Because of equality 
\[
\dim_k A = \dim_k H -\dim_k \ci
\]
it suffices to show
\begin{align} \label{IE:I}
\dim_k \ci \geq \dim_k \Ext^2_A(K, K) + \ldots + \dim_k \Ext^n_A(K, K)
\end{align}
By Theorem~\ref{T:Bongartz},
we have
\begin{align*}
\dim_k \Ext^{2l}_A(K, K) = \dim_k(\ci^l \cap \mathfrak{r}_H\ci^{l-1}\mathfrak{r}_H)  - \dim_k(\ci^l \mathfrak{r}_H +\mathfrak{r}_H\ci^l)\\
\dim_k \Ext^{2l+1}_A(K, K)=\dim_k(\mathfrak{r}_H\ci^l \cap \ci^{l}\mathfrak{r}_H) - \dim_k(\ci^{l+1} +\mathfrak{r}_H\ci^l\mathfrak{r}_H), 
\end{align*}
showing that the right hand side of \eqref{IE:I} equals (assuming that $n=2l+1$ is odd---the case that $n$ is even is analogous)
\begin{align}
&\dim_k \ci \\
&- \dim_k(\ci \mathfrak{r}_H +\mathfrak{r}_H\ci) + \dim_k(\ci \mathfrak{r}_H \cap \mathfrak{r}_H\ci) \label{E:L2} \\
	       &- \dim_k(\ci^{2} +\mathfrak{r}_H\ci\mathfrak{r}_H) + \dim_k(\ci^2 \cap \mathfrak{r}_H\ci\mathfrak{r}_H) \\
	       & \cdots \\
	       &   - \dim_k(\ci^l \mathfrak{r}_H +\mathfrak{r}_H\ci^l) + \dim_k(\ci^l \mathfrak{r}_H \cap \mathfrak{r}_H\ci^{l}) \\
	        &- \dim_k(\ci^{l+1} +\mathfrak{r}_H\ci^l\mathfrak{r}_H). \label{E:Ll}
	     \end{align}
All the contributions in \eqref{E:L2} -- \eqref{E:Ll} are $\leq 0$ -- indeed,  
\begin{align*}
\ci \mathfrak{r}_H +\mathfrak{r}_H\ci &\supseteq \ci \mathfrak{r}_H \cap \mathfrak{r}_H\ci \\
\ci^{2} +\mathfrak{r}_H\ci\mathfrak{r}_H &\supseteq \ci^2 \cap \mathfrak{r}_H\ci\mathfrak{r}_H \\
&\cdots
\end{align*}
This completes the proof.
\end{proof}

\begin{ex} \label{Ex:dimEstimateLongerPaths}
Let $Q$ be the following quiver
\begin{equation*}\begin{tikzpicture}[description/.style={fill=white,inner sep=2pt}]
   \matrix (m) [matrix of math nodes, row sep=3em,
                 column sep=2.5em, text height=1.5ex, text depth=0.25ex,
                 inner sep=2pt, nodes={inner xsep=0.3333em, inner
ysep=0.3333em}] at (0, 0)
    {  1&2&3& \cdots &n. \\};
 \path[->] ($(m-1-1.east)$)   edge  node [fill=white, scale=0.75] [midway] {$a_1$}  ($(m-1-2.west)$) ;
 \path[->] ($(m-1-2.east)$)  edge  node [fill=white, scale=0.75] [midway] {$a_2$}  ($(m-1-3.west)$) ;
 \path[->] ($(m-1-3.east)$)  edge  ($(m-1-4.west)$) ;
 \path[->] ($(m-1-4.east)$)  edge  ($(m-1-5.west)$) ;
 \end{tikzpicture}\end{equation*}
 \begin{enumerate}
 \item Let $A=kQ/(a_2a_1)$. We can use the formulas of Theorem~\ref{T:Bongartz} to compute the dimensions of the $\Ext$-groups. The left and right hand sides of Corollary \ref{C:IE} are   
 \[
 \dim_k A= \dim_k kQ - \dim_k (a_2a_1) = {\binom{n+1}{2}} - (n-2) 
\]
and
\begin{multline*}
 \dim_k kQ - \sum_{l=2}^m \dim_k \Ext^l_A(K, K) \\
=  \dim_k kQ - \dim_k \Ext^2_A(K, K) = {\binom{n+1}{2}} - 1. 
 \end{multline*}
 This shows that their difference can become arbitrarily large.\\

 \item Let $B=kQ/(\rad kQ)^2$. This algebra has global dimension $n-1$ and we get an equality in Corollary \ref{C:IE}:
  \begin{align*}
 \dim_k B&= 2n -1 = {\binom{n+1}{2}} - \sum_{l=2}^{n-1} (n-l) \\ &=\dim_k kQ - \sum_{l=2}^{n-1} \dim_k \Ext^l_B(K, K)
 \end{align*}
 The latter statement generalises to quivers $Q$ of the form 
 \begin{equation*}\begin{tikzpicture}[description/.style={fill=white,inner sep=2pt}]
   \matrix (m) [matrix of math nodes, row sep=3em,
                 column sep=2.5em, text height=1.5ex, text depth=0.25ex,
                 inner sep=2pt, nodes={inner xsep=0.3333em, inner
ysep=0.3333em}] at (0, 0)
    {  && 11 & \ldots & l_11\\
        && 12   & \ldots & l_22 \\
       0&&&\ldots &&&  1 \\
       && 1n & \ldots & l_nn \\
    };

 \path[->] ($(m-3-1.east)$)  edge   ($(m-1-3.west)$) ;
 \path[->] ($(m-3-1.east) $) edge   ($(m-2-3.west)$) ;
 \path[->] ($(m-3-1.east) $) edge   ($(m-3-4.west)$) ;
  \path[->] ($(m-3-1.east) $) edge   ($(m-4-3.west)$) ;
  
   \path[<-] ($(m-1-4.west)$)  edge   ($(m-1-3.east)$) ;
 \path[<-] ($(m-2-4.west) $) edge   ($(m-2-3.east)$) ;
  \path[<-] ($(m-4-4.west) $) edge   ($(m-4-3.east)$) ;
  
   \path[->] ($(m-1-4.east)$)  edge   ($(m-1-5.west)$) ;
 \path[->] ($(m-2-4.east) $) edge   ($(m-2-5.west)$) ;
  \path[->] ($(m-4-4.east) $) edge   ($(m-4-5.west)$) ;

 \path[->] ($(m-1-5.east)$)  edge   ($(m-3-7.west)+(0, 2mm)$) ;
 \path[->] ($(m-2-5.east) $) edge   ($(m-3-7.west)+(0, 1mm)$) ;
 \path[->] ($(m-3-4.east) $) edge    ($(m-3-7.west)$) ;
  \path[->] ($(m-4-5.east) $) edge   ($(m-3-7.west)+(0, -1mm)$) ;

 \draw[ ->] ($(m-3-1.east) + (0mm,0mm)$) .. controls +(2.5mm,-35mm) and
+(-2.5mm,-35mm) ..  ($(m-3-7.west) + (0mm,-2mm)$);
 \end{tikzpicture}\end{equation*}
 
 \end{enumerate}
\end{ex}

%% file: kapitel3.tex
\section{Saturation}\label{sec3}

Throughout the chapter, we fix a category $\Ch$ with the following properties:
\begin{setup}\label{setup:artin}
Let $k$ be a perfect field and $\Ch$ an essentially small $k$-linear abelian category such that:
\begin{enumerate}
\item $\Ch$ is artinian and noetherian;
\item all $\Hom$-spaces are finite dimensional;
\item the set of simple objects of $\Ch$ admits a partial order such that
\[ \Ext^1(S,S')\neq 0 \text{\ only if\ }S<S';\]
\item $\Ch$ is hereditary, i.e. $\Ext^{>1}_\Ch(-, -)=0$. 
\end{enumerate}.
\end{setup}

Let $\Bh\subset\Ch$ be a full subcategory. We want to construct its \emph{hereditary closure} or \emph{saturation} $\Bh^\sat$, i.e., a minimal subcategory of $\Ch$ containing $\Bh$ which is in addition hereditary. 

\begin{thm}\label{thm:sat_exists}
Given $\Ch$ as  above and $\Bh\subset\Ch$ a full abelian subcategory closed under subquotients. Then there is a minimal full hereditary subcategory closed under subquotients $\Bh^\sat$
\[ \Bh\subset\Bh^\sat\subset\Ch\]
containing $\Bh$. 
\begin{enumerate}
\item \label{it:1thm} $\Bh$ and $\Bh^\sat$ have the same class of simple objects.
\item\label{it:2thm}
For all simple objects $S,T$ of $\Bh$, we have
\[ \Ext^1_\Bh(S,T)=\Ext^1_{\Bh^\sat}(S,T).\]
\item \label{it:3thm}The subcategory is maximal with these properties.
\end{enumerate}
\end{thm}

\begin{rem}
The hereditary closure is \emph{not} canonical, see Example~\ref{ex:not_can} below, even if property \eqref{it:3thm} suggests otherwise. There is no global maximum in general. 
\end{rem}

\begin{rem}We have two applications for this abstract result. It allows us to reduce (conjecturally) questions about relations between period numbers to the hereditary case, see Theorem~\ref{thm:dim_formula} and Remark~\ref{rem:reduce}. This becomes unconditional in the cases of $1$-motives and Mixed Tate Motives. In the latter setting, we will apply it in Lemma~\ref{lem:construct_tate} in order to construct hereditary categories of Mixed Tate Motives with finite dimensial $\Ext$-groups. This recovers a construction of Deligne--Goncharov in \cite{deligne-goncharov}, who used Tannakian methods.
\end{rem}

The proof will be given towards the end of the chapter.
Recall from Definition~\ref{defn:str_finitary} the notion of a strongly finitary category.

\begin{lemma}\label{lem:limits}
Let $\Ch$ as above. There exists a directed set $I$ and
a direct system $\Ch_i$ for $i\in I$ of strongly finitary full subcategories of $\Ch$ closed under subquotients such that 
 $\Ch=\varinjlim_{i\in I} \Ch_i$.

Let $S,T$ be simple objects of $\Ch$. Then
\[ \varinjlim_{i\in I}\Ext^1_{\Ch_i}(S,T)=\Ext^1_\Ch(S,T)\]
with injective transition maps. Moreover,
\[ \varinjlim_{i\in I}\Ext^2_{\Ch_i}(S,T)=0.\]
\end{lemma}
\begin{proof} 
Recall that for an object $M$ in $\Ch$, we denote by $\langle M\rangle$ the full abelian subcategory of $\Ch$ containing $M$ and closed under subquotients. Its objects are the subquotients of $M^n$ for all $n\geq 1$.
These categories form a set. The system is directed because both
$\langle M\rangle$ and $\langle M'\rangle$ are contained in $\langle M\oplus M'\rangle$. Each $\Ch_i$ has only finitely many simple objects (up to isomorphism) because all objects of $\Ch$ have finite length. All $\Hom$-groups are finite dimensional because this holds in $\Ch$. By Corollary~\ref{cor:is_str_finitary} this makes $\Ch_i$ strongly finitary. By Proposition~\ref{P:Deligne} its $\Ext^1$'s are finite dimensional.

We spell out the details of the proof
for $\Ext^2$. The proof of the statement on $\Ext^1$ is analogous. 
As $\Ch$ is hereditary, its $\Ext^2$s vanish. 
On the other hand, we can compute via the  limit.  Each element of $\varinjlim\Ext^2_{\Ch_i}(S,T)$ is represented by an exact sequence
\[ 0\to T\to X_2\to X_1\to S\to 0\]
in $\Ch_i$ for some $i\in I$.
As $\Ch$ is hereditary, the sequence is trivial in $\Ch$. Hence  there is a commutative diagram of exact sequences in $\Ch$
\[ \xymatrix{
0\ar[r]& T\ar[r] &X_2\ar[r]& X_1\ar[r]& S\ar[r] &0\\
0\ar[r]& T\ar[u]^{=}\ar[r] &Y_2\ar[r]\ar[u]& Y_1\ar[r]\ar[u]& S\ar[u]^{=}\ar[r]& 0
}
\]
and a retraction $Y_2\to T$. All this data lives in $\Ch_j$ for some $j\geq i$.
\end{proof}

\begin{prop}\label{prop:choice}
Let $\Ch$ be as above, $\Bh\subset\Ch$ a strongly finitary full subcategory closed under subquotients. Then $\Bh$ is equivalent to the category of representations of a $k$-algebra $k\cs/\Ih$ where $\cs$ is the species of $\Bh$ with semi-simple quotient $K$ and $\Ih$ a finite dimensional two-sided ideal
generated by $D\Ext^2_{\Bh}(K,K)$.

The ideal $\Ih$ and the inclusion of $D\Ext^2_\Bh(K,K)$ into $\Ih$ can be chosen functorially in $\Bh$. 
\end{prop}

\begin{proof} See Appendix~\ref{app}.
\end{proof}

\begin{rem}There is an overlap with the results of Iusenko--MacQuarrie in \cite{iusenko-macquarrie}. They consider certain algebras $A$ over perfect fields $k$ (in particular finite dimensional ones) and construct the surjection
$k\cs\to A$ functorially up to some equivalence. This is not enough for our purposes. We thank Julian Külshammer for making us aware of the reference.
\end{rem}

From now on we fix functorial choices as in the proposition.
In order to prove the theorem, we need to show that the hereditary category 
$\modules{k\cs_\Bh}$ 
can be realised as a subcategory of $\Ch$.

\begin{lemma}\label{lem:last}
For each strongly finitary full subcategory $\Bh\subset\Ch$ closed under subquotients there exists
a strongly finitary $\Bh\subset\Bh'\subset\Ch$ closed under subquotients such that the map
\[ k\cs_{\Bh'}/\Ih'\to k\cs_{\Bh}/\Ih\]
factors via $k\cs_\Bh$. 

The category $\modules{k\cs_\Bh}$ is isomorphic to a full subcategory of $\Bh'$.
\end{lemma}
\begin{proof} Let $I$ be the direct system of full strongly finitary subcategories
$\Bh_i\subset\Ch$ closed under subquotients exhausting $\Ch$.
	We write $H_i=k\cs_i$ for the path algebra of the species for $\Bh_i$. It is finite dimensional because $\cs_i$ does not contain any oriented cycles. Let $K_i=H_i/\rad(H_i)$.
We consider
\begin{equation}\label{eq:lim} 0\to \Ih_i\to H_i\to H_i/\Ih_i\to 0\end{equation}
such that $\Bh_i\isom\modules{H_i/\Ih_i}$.
	Moreover, the ideal $\Ih_i$ is generated by the image of $D\Ext^2_{\Bh_i}(K_i,K_i)$. Note that
the transition map $H_j\to H_i$ is an epimorphism whenever $j\geq i$ by Remark~\ref{rem:is_surj}. 

Recall that $\Ext^2_{\Bh_i}$ is finite-dimensional and $\varinjlim\Ext^2_{\Bh_i}$ vanishes by Lemma~\ref{lem:limits}. Hence for all $i$ and simple objects $S,T$ of $\Bh_i$, there is $j\geq i$ such that
	the image of $\Ext^2_{\Bh_i}(S,T)$ in $\Ext^2_{\Bh_j}(S,T)$ vanishes. As there are only finite many simple objects in $\Bh_i$, the index $j$ can be chosen for all $S,T$ in $\Bh_i$ simultanuously. The direct system $\Ext^2_{\Bh_i}(K_i,K_i)$
is ML-zero. This makes the dual projective system $D\Ext^2_{\Bh_i}(K_i,K_i)$ AR-zero. 
	In turn this implies that the system $(\Ih_i)_{i\in I}$ is AR-zero. 
From this we get for each $i$ an index $j\geq i$ such that 
\[ H_j\twoheadrightarrow H_i\]
factors via $H_j/\Ih_j$ and
\[ H_{j}/\Ih_{j}\to H_i/\Ih_i\]
factors via $H_i$. For $\Bh=\Bh_i$ and setting $\Bh'=\Bh_j$ this was the claim. 

The factorisation $k\cs_{\Bh'}/\Ih'\twoheadrightarrow k\cs_\Bh\twoheadrightarrow k\cs_\Bh/\Ih$ makes $k\cs_\Bh$ an object of $\modules{k\cs_{\Bh'}/\Ih'}$, and hence an object $X$ of $\Bh'$. The category
$\langle X\rangle$ is equivalent to $\modules{k\cs_\Bh}$. 
\end{proof}

\begin{rem} All terms of the short exact sequence \eqref{eq:lim} are finite dimensional. This make $\varprojlim$ exact. The argument in the proof shows that
\[ H:=\varprojlim_i H_i/\Ih_i\isom\varprojlim_i H_i.\]
The algebra $H$ is pseudo-compact in the sense of \cite{gabriel}. We find $\Ch$ as the category of its discrete representations.
\end{rem}

\begin{proof}[Proof of Theorem~\ref{thm:sat_exists}.]
By Lemma~\ref{lem:limits}, the subcategory $\Bh$ is a limit of full strongly finitary subcategories $\Bh_i$ closed under subquotients. If the theorem holds for
	all $\Bh_i$, then we may put $\Bh^\sat=\varinjlim_i\Bh_i^\sat$. It has the required properties because $\Ext^n$ commutes with the direct limit. 

Hence it suffices to consider the case $\Bh=\langle M\rangle$ for an object $M\in \Ch$. It is the category of $H/\Ih$-modules with $H=k\cs$ hereditary and $\Ih$ admissible. We define  $\Bh^\sat$
	as the category of f.g. $H$-modules. By Lemma~\ref{lem:last} it is equivalent to a full subcategory of $\Ch$. By construction it satisfies \eqref{it:1thm} (same simple objects) and \eqref{it:2thm} (same $\Ext^1$) of the Theorem. 

Let $\Bh^\sat\subset\Bh'\subset \Ch$ be a full abelian subcategory satisfying
 properties \eqref{it:1thm} and \eqref{it:2thm} of Theorem~\ref{thm:sat_exists}. 
The category $\Bh'$ has the same species as $\Bh$, hence $\Bh^\sat\subset\Bh'^\sat$ is an inclusion of hereditary categories with the same species. This implies
that they are equal, and hence also $\Bh^\sat=\Bh'$.
\end{proof}

%%%%%%%%%%%%%%%%%%%%%%%%%
The above proof gives the following explicit description of the saturation:
\begin{cor}\label{cor:Msat}
 Let $\Ch$ be as above. If $\Bh=\langle M\rangle\subset\Ch$ is generated by a single object, then so is $\Bh^\sat$. The saturation is equivalent to the category
$\modules{k\cs}$ where $\cs$ is the species of $\Bh$.
\end{cor}

\begin{cor}\label{cor:finite_ext}
Let $S_1,\dots,S_n$ be simple objects of $\Ch$ and for each $i\neq j$ fix a finite dimensional subspace
	$V_{ij}\subset\Ext^1_{\Ch}(S_i,S_j)$ which is an $\End(S_i)-\End(S_j)$-sub-bimodule.
Then there is a hereditary category $\Bh\subset\Ch$ with these simple objects and $V_{ij}=\Ext^1_{\Bh}(S_i,S_j)$. 

If, moreover,  the partial order on the set of simple objects of $\Ch$ 
of  Set-Up~\ref{setup:artin} can be refined to a total order such that
$V_{ij}=\Ext^1_\Ch(S_i,S_j)$ unless $i$ and $j$ are neighbours in the total order, then $\Bh$
 is given by the full subcategory of objects $X$ such that all subquotients of $X$ of length $2$  which are extensions of the simple objects $S_i$ and $S_j$ are in $V_{ij}$.
\end{cor}
\begin{proof} We pick $k$-bases of the $V_{ij}$ and for each basis element an
object of $\Ch$ representing the extension. Let $M$ be the direct sum of all these objects and $\Bh$ the saturation of $\langle M\rangle$ in $\Ch$. It is hereditary and has the right $\Ext^1$.

Now assume that we have a total order $S_1<S_2<\dots <S_n$ and $V_{ij}$ as in the statement. 
Without loss of generality, these are the only simple objects of $\Ch$.
For every object $X$ of $\Ch$ and $a\in \{1,\dots,n \}$ there is a canonical short exact sequence
\[ 0\to X_{[a,n]}\to X\to X_{[1,a-1]}\to 0\]
such that $X_{[a,n]}$ has simple subquotients in $\{S_a,\dots,S_n\}$ and
$X_{[1,a-1]}$ has simple subquotients in $\{S_1,\dots,S_{a-1}\}$. For $1\leq a\leq b\leq n$ we define
$X_{[a,b]}$ as the kernel of $X_{[a,n]}\to X_{[1,b-1]}$. It captures the simple subquotients in $\{S_a,\dots,\dots,S_b\}$. 

Note that $X_{[a-1,a]}$
is a sum of elements of $\Ext^1_{\Ch}(S_{a-1},S_a)$. Assume that is
is in $V_{a-1a}$ for all $a$ and hence $X_{[a-1,a]}\in \Bh$. We claim
that $X_{[a,b]}$ is in $\Bh$ for all $a\leq b$ and argue by induction on
$b-a\geq 0$. The cases $b-a=0,1$ hold by assumption.

Suppose that $X_{[a,b]}\in\Bh$ for $b-a\leq m$. Let $i<a$. We have a commutative diagram of exact sequences
\[\begin{xy}\xymatrix{
0\ar[r]&\Ext^1_\Bh(S_i,X_{[a+1,b]})\ar[d]\ar[r]&\Ext^1_\Bh(S_i,X_{[a,b]})\ar[r]\ar[d]& \Ext^1_\Bh(S_i,X_{[a,a]})\ar[d]\ar[r]&0\\
0\ar[r]&\Ext^1_\Ch(S_i,X_{[a+1,b]})\ar[r]&\Ext^1_\Ch(S_i,X_{[a,b]})\ar[r]&\Ext^1_\Ch(S_i,X_{[a,a]})\ar[r]&0
}\end{xy}
\]
If $i<a-1$, then the horizontal map on the right is an isomorphism by our assumption on $V_{i,a}$. By a second induction on $b-a$, this implies that all horizontal maps are isomorphisms in this case. In particular
\[ \Ext^1_{\Bh}(X_{[a-1,a-1]},X_{[a+1,b]})\isom \Ext^1_{\Ch}(X_{[a-1,a-1]},X_{[a+1,b]}).\]
We now consider $X_{[a-1,b]}$. It defines
an element of $\Ext^1_\Ch(X_{[a-1,a-1]},X_{[a,b]})$ whose image in 
$\Ext^1_\Ch(X_{[a-1,a-1]},X_{[a,a]})$ is in $\Ext^1_\Bh(X_{[a-1,a-1]},X_{[a,a]})$. By a little diagram chase, this information implies that $X_{[a-1,b]}$ is 
in $\Bh$.
\end{proof}
\begin{rem}A similar process is applied by Deligne and Goncharov \cite{deligne-goncharov} in order to define mixed Tate motives over the rings of integers as a full 
subcategory of all mixed Tate motives over a number field. They use Tannakian methods instead. We will come back to this point in Chapter~\ref{sec:MTM}.
\end{rem}

%%%%%%%%%%%%%%%%%%%%%%%%%%%%%%%%%%%%%%%%

\begin{ex}\label{ex:not_can}
As promised, we show that the hereditary closure is not canonical as a subcategory. Let $k$ be a perfect field, $Q$ the quiver with vertices $x,y,z$ and
edges $\alpha:x\to z$, $\beta:x\to y$, $\gamma:y\to z$. Let $\Ch$ be the
category of finite dimensional representations of $Q$. 

We denote
by $M_\alpha$ the representation sending $x,z$ to $k$, $y$ to $0$ and $\alpha$ to $1$. It represents the generator of $\Ext^1_\Ch(S(z),S(x))$ where
$S(x), S(z)$ are the simple objects corresponding to $x$ and $z$. We define
 $M_\beta, M_\gamma$ analogously. Finally, for $r\in k$, denote by $M_r$
the representation mapping $x,y,z$ to $k$, $\beta,\gamma$ to $1$ and $\alpha$ to $r$.

Let $\Bh=\langle M_\beta\oplus M_\gamma\rangle$. 
We have
\[ \Bh\subset \langle M_r\rangle.\]
We claim that $\langle M_r\rangle$ is a hereditary closure. By Lemma~\ref{L:DimFromGen}, the category is equivalent to the category of representations of
\[ kQ/\langle r\gamma\beta-\alpha\rangle \isom kQ'\]
where $Q'\subset Q$ is the subquiver omitting the edge $\alpha$. The algebra is hereditary and has the required simple objects and Ext-groups. Note that this is
true for every $r$. These categories are distinct because
\[ \langle r\gamma\beta-\alpha\rangle \cap \langle r'\gamma\beta-\alpha\rangle=0\]
for all $r\neq r'$.

Note that $\Bh$ does not satisfy the conditions on Ext-groups in Corollary~\ref{cor:finite_ext}. The total order is $x>y>z$ but  
\[ 0=V_{xz}\neq \Ext^1_\Ch(S(x),S(z))\isom k  \ .\] 
\end{ex}

%% file: kapitel4.tex
\section{Periods of motives}\label{sec:periods}
We introduce the period space of a motive and formulate dimension formulas for them. They are sharp in the cases where the period conjecture is known and upper bounds in general. For a more complete introduction, we refer to the survey article \cite{huber_galois}.

Let $F\subset\Qbar\subset\C$ be a field of algebraic numbers. 
%We fix an embedding $F\to\C$.
Let $\MM$ be a $\Q$-linear category of mixed motives over $F$. To fix
ideas we can choose it as a (not necessarily full) subcategory of Nori's
category of mixed motives, see \cite{period-buch}.

Motives are supposed to capture homological properties of the category of
algebraic varieties over $F$ (in our case). We are interested in two such theories.

Let 
\begin{align*}
\hsing:\MM&\to\modules{\Q},\\
\hdR:\MM&\to\modules{F}
\end{align*}
be the singular and de Rham realisation, respectively. The functors are $\Q$-linear, faithful and exact. Moreover, there is a functorial isomorphism
\begin{equation}\label{per} \pi:\hsing\tensor_\Q\C\isom\hdR\tensor_F\C.\end{equation}
For the constructions, we refer to \cite{period-buch}. For our purposes, it suffices to use them as a black box. 
\begin{lemma}\label{lem:is_ok}
Let $\MM$ be a $\Q$-linear category as above. Then $\MM$ satisfies the assumptions of Set-Up~\ref{setup:artin} (1)--(3).
\end{lemma}
\begin{proof}Conditions (1) and (2) hold because the category $\MM$ has an exact faithful functor $H_\sing$ to the category of finite dimensional vector spaces $\modules{\Q}$, which satisfies the properties (1) and (2). Every motive carries a canonical weight filtration, see \cite[Chapter~10]{period-buch}. 
If $S,S'$ are simple of weights $w,w'\in\Z$, respectively, then the weight filtration splits every element of $\Ext^1_{\MM}(S,S')=0$ unless $w>w'$. This is the partial order required in (3).
\end{proof}

The special choices of categories $\MM$ (of $1$-motives and of Mixed Tate Motives, respectively) that we treat in the next two chapters are known to be hereditary. Moreover, the category of all mixed motives over fields contained in $\Qbar$ is conjectured to be hereditary. 
Note, however, that full abelian subcategories of hereditary categories need not be hereditary in general. A $2$-extension in the subcategory can be split by an object of the  bigger category. So the categories $\MM$ are not necessearily hereditary in general.

\begin{defn}\label{defn:period-space}
Let $M\in\MM$. The \emph{period space of $M$} is the $F$-vector space $\Per\langle M\rangle$ generated
by the entries of a matrix for the isomorphism $\pi$ in Equation (\ref{per}) in a $\Q$- and $F$-basis, respectively. Its elements are called \emph{periods of $M$}.
\end{defn}

The period matrix itself depends on the choice of bases. However, passing to different bases changes the period matrix only by multiplication by an invertible matrix with $\Q$- or $F$-coefficients, respectively. Thus  the vector space $\Per\langle M\rangle$ is independent of the choice of bases. Transcendence theory asks about linear or algebraic relations between elements of period spaces. The \emph{period conjecture} says qualitatively that all such relations are induced from a small set of obvious relations, indeed from functoriality of motives. We have the language to make this more precise.

\begin{defn}\label{defn:Nori-alg}
Let $M\in\MM$. The \emph{Nori algebra} $A(M)$ is defined
as the ring of endomorphisms of the functor
\[ \hsing:\langle M\rangle\to \modules{\Q}.\]
\end{defn}
The algebra $A(M)$ is finite dimensional over $\Q$. By construction 
$\hsing$ can be understood as a functor to $\modules{A(M)}$.
Nori proved, see also \cite[Section~7.3]{period-buch} that $\langle M\rangle$ is 
equivalent to the category of $A(M)$-modules via this functor.

\begin{defnthm}[{\cite{huber_galois}}]The \emph{period conjecture holds for $M$} if and only if
\[ \dim_\Q A(M)=\dim_F \Per\langle M\rangle.\]
The estimate $\geq$ holds unconditionally.
\end{defnthm}
In our context, we use this as a definition, but it is actually a theorem that the statement is equivalent to the Nori-Kontsevich formulation of the period conjecture.

We now turn to a quantitative version of the conjecture where we give a formula for $\dim_F\Per\langle M\rangle$.

\begin{notation}\label{not:mult}
Fix $M\in\MM$. Let $M_1,\dots,M_n$ be the simple objects of $\langle M\rangle$. We denote
\begin{itemize}
\item $D_i=\End_\MM(M_i)$ a skew field of $\Q$-dimension $d_i$;
\item $m_i=\dim_{D_i}\hsing(M_i)$;
\item $\cs=((D_i)_i, ({}_iE_{j})_{ij} ) $ the species of $\langle M\rangle$ with vertices $1,\dots,n$, see Definition~\ref{defn:species_of_c};
\item for a path $\gamma:s\pfad s'$ in $\cs$ let $E(\gamma)$ be the tensor product of the $D\Ext^1_{\langle M\rangle}(M_{\gamma(t)},M_{\gamma(t+1)})$ along the path, see Definition~\ref{defn:q_from_cs}.
\end{itemize} 
\end{notation}

\begin{lemma}\label{lem:mult}
The number $m_i$ agrees with the multiplicity of the projective cover $P_i$ of the  simple object $M_i$ in $A(M)$ in the sense of Section~\ref{ssec:not_formulas}.
\end{lemma}
\begin{proof}As mentioned above, $A(M)$ is defined as the endomorphisms of the functor
\[ H_\sing:\langle M\rangle\to \modules{\Q}.\]
The functor tautologically takes values in the category $\modules{A(M)}$. It is an equivalence of categories.  This implies
\[ H_\sing(X)=\Hom_{A(M)}(A(M),\hsing(X))\]
for all $X\in\langle M\rangle$.
By Equation \eqref{E:decomp} of Section~\ref{ssec:morita} we decompose
\[ A(M)=\bigoplus P_j^{m_j}\]
 where $P_j$ is the projective cover of the simple object $\hsing(M_j)$. This implies
\begin{multline*}
 H_\sing(M_i)=\Hom_{A(M)}\left(\bigoplus P_j^{m_j},\hsing(M_i)\right)\\
=\Hom_{A(M)}(P_i^{m_i},\hsing(M_i))
=\End_{A(M)}(\hsing(M_i))^{m_i}
\\ 
=\End_{\langle M\rangle}(M_i)^{m_i}
\end{multline*}
and hence
\[ \dim_\Q H_\sing(M_i)=m_id_i\]
agrees with the characterisation of $m_i$ above.
\end{proof}

\begin{thm}[Dimension formula]\label{thm:dim_formula} Let $M\in\MM$ and consider $\langle M\rangle$. Let $A(M)$ be the Nori algebra. 
\begin{enumerate}
\item
Using Notation~\ref{not:mult} we have
%\begin{equation}\label{eq:formula} \dim_\Q A(M)\leq\dim_\Q A(M^\sat)= \sum_{i=1}^n m^2_id_i + \sum_{i,j=1}^n\sum_{\gamma:i\pfad j}m_im_j\dim_\Q E(\gamma)
%\end{equation}
\begin{equation}\label{eq:formula} \dim_\Q A(M)\leq\sum_{i=1}^n m^2_id_i + \sum_{i,j=1}^n\sum_{\gamma:i\pfad j}m_im_j\dim_\Q E(\gamma)
\end{equation}
(where the sum is taken over all paths $\gamma$ of positive length in $\cs$)
and
\[ \dim_\Q E(\gamma)=\frac{\prod\limits_{\text{$\varepsilon:v\to w$ on $\gamma$}}\dim_\Q {}_wE_v}{\prod\limits_{v\in\gammacirc} d_v}\]
(where $\gammacirc$ is the set of vertices on the path which are different from the end points).

\item \label{it:formula_2}We have equality if and only if $\langle M\rangle$ is hereditary.

\item Assume, moreover, that $\MM$ is hereditary.
Then there is $M^\sat\in\MM$  such that $\langle M^\sat\rangle=\langle M\rangle^\sat$ is a hereditary closure of $\langle M\rangle$. In this case the right hand side of \eqref{eq:formula} is equal to $\dim_\Q A(M^\sat)$.
\end{enumerate}
\end{thm}
\begin{proof}
The formula \eqref{eq:formula} follows from Proposition~\ref{L:dimPathAlg}. Indeed, the first sum collects the paths of length zero (i.e., the vertices). The second sum collects the paths of positive length.

Also by Proposition~\ref{L:dimPathAlg} we have equality if and only if
$A(M)$ is hereditary.

Finally, assume $\MM$ is hereditary. By Lemma~\ref{lem:is_ok}, 
we can apply the results of Chapter~\ref{sec3}. By Theorem~\ref{thm:sat_exists}, the hereditary closure
$\langle M\rangle^\sat$ exists. The existence of $M^\sat$ is Corollary~\ref{cor:Msat}. 
We have
$\langle M\rangle\subset\langle M^\sat\rangle$ and hence (or by construction)
$A(M^\sat)\twoheadrightarrow A(M)$. 

By the explicit description of $\langle M\rangle^\sat$ in Theorem~\ref{thm:sat_exists}, the data of Notation~\ref{not:mult} agrees for $M$ and $M^\sat$. Hence the same is true for the right hand sides of formula \eqref{eq:formula}. 
We have equality in \eqref{eq:formula} in case of $M^\sat$ by item~(\ref{it:formula_2}).
\end{proof}

\begin{rem}
\begin{enumerate}
\item
This confirms the expectation of Nesa in \cite[Expectation~3.4.3]{nesa}, who considered the case of an extension of two simple objects.
\item The estimates of Corollary~\ref{C:IE} can be used to sharpen the estimates if $\langle M\rangle$ is not hereditary.
\end{enumerate}
\end{rem}

\begin{thm}%\label{thm:intro}
The dimension estimate for $\dim_F\Per\langle M\rangle$ formulated in equation \eqref{eq:intro} in the introduction holds true.
If the Period Conjecture holds for $M$, then we have equality in the formula
in the two cases
\begin{itemize}
\item $\langle M\rangle$ is hereditary, i.e., all $\Ext^k_{\langle M \rangle}$ vanish for $k \geq 2$; or
\item  the maximal length of a path in the quiver of the species of $\langle M \rangle$ is $2$,
in particular $\Ext^k_{\langle M \rangle}$ vanishes for all $k \geq 3$.
\end{itemize}
\end{thm}
\begin{proof}
We have unconditionally 
\[ \dim_F\Per\langle M\rangle\leq \dim_\Q A(M)\]
with equality if and only if the period conjecture holds for $M$.
We combine this with the dimension formulas for $A(M)$ in Proposition~\ref{L:dimPathAlg} and Corollary~\ref{C:IE}.

If the period conjecture holds and $\langle M\rangle$ is hereditary, then we get equality by 
Theorem~\ref{thm:dim_formula}. If the period conjecture holds and the maximal length of a path in the quiver of the species is $2$, then we get equality by
 Corollary~\ref{c: lengthatmost2}.
\end{proof}

\begin{rem}\label{rem:reduce}
If $\MM$ is hereditary, then $\Per\langle M\rangle\subset \Per\langle M^\sat\rangle$. Expressing the elements of $\Per\langle M\rangle$ in a basis of
$\Per\langle M^\sat\rangle$ would allow us to determine all linear relations between them. 
Questions about periods can thus be reduced to the hereditary case.
Note, however, that this approach does not give a simple dimension formula, but rather an algorithm in each case.
\end{rem}

\begin{rem}\label{rem:brown}
We explain the relation with Brown's algorithmic approach in \cite{brown-phys} that applies in special cases but does not produce explicit dimension formulae.
He takes the Tannakian point of view, which has been prominent in most of the literature linking motives and the Period Conjecture. (The relation to our formulation is, for example, spelled out in \cite{huber_galois}.) The \emph{motivic Galois group} $G_\mot(M)$ of $M$ is the Tannakian dual of the rigid tensor category $\langle M\rangle^\tensor$ generated by $M$. It is an algebraic group over $\Q$. The Period Conjecture relates the Hopf algebra $\Oh(G_\mot(M))$ to the algebra generated by the periods of $M$. The coalgebra $DA(M)$ dual to our $A(M)$ agrees with  the subspace of matrix coefficients
of $\Oh(G_\mot(M))$, see \cite[Proposition~4.2.2]{nesa}.

Under the assumption that $\langle M\rangle^\tensor$ is hereditary, Brown explains an algorithm to describe the linear relations between the elements of the Hopf algebra of the unipotent radical of $G_\mot(M)$ in terms of the coradical filtration. 
 It remains open how to combine this with information on the reductive quotient. In fact, the examples of Nesa in \cite[Chapter~5]{nesa} illustrate that the linear spaces of periods are not well-behaved under the decomposition of the algebraic group into unipotent radical and reductive quotient.
\end{rem}

%% file: kapitel5.tex
\section{Periods of \texorpdfstring{$1$}{1}-motives}\label{sec:1-mot}
In this section we specialise our considerations to
$\MM=\onemot$, the category of iso-$1$-motives over $F=\Qbar$; see Section~\ref{ssec:cat_1} for a more detailed description. 
This improves and clarifies the partial results of \cite[Part~Four]{huber-wuestholz}, which were formulated in ad hoc terms specific to the situation.

We will only recall those facts about $1$-motives that are needed to apply our dimension formulas. 
 This category was introduced by Deligne in \cite{hodge3} in order to capture the homological properties of algebraic varieties in homological degree $1$. Actually, their periods (also called $1$-periods) are the same as the periods of curves, see \cite[Summary~12.11]{huber-wuestholz}. 

The category is $\Q$-linear, abelian and hereditary by \cite[Prop.~3.2.4]{orgogozo}. The weight filtration on  $\onemot$ is concentrated in degrees $-2,-1,0$. This implies that the paths in the (infinite) quiver of the species of $\onemot$ have length at most  $2$. As before, we write $\langle M\rangle$ for the full subcategory of $\onemot$ closed under subquotients and containing $M$.  We denote its Nori algebra $A(M)$, see Definition~\ref{defn:Nori-alg}. As pointed out above,
\[ \langle M\rangle \isom \modules{A(M)}.\]

\begin{thm}Let $M\in\onemot$. Let $M^\ss$ be the sum of the simple subquotients of $M$ without multiplicities. 
Then $\langle M\rangle$ has global dimension at most $2$. 
It is equivalent to the category of finite dimensional modules over the finite dimensional $\Q$-algebra $B/\ci$
\[ \langle M\rangle\isom \modules{B/\ci}\]
 where $B$ is the tensor algebra of $E:=D\Ext^1_{\langle M\rangle}(M^\ss, M^\ss)$ over $\End(M^\ss)$ 
\[
   B=\End(M^\ss)\oplus E\oplus E\tensor_{\End(M^\ss)} E\]
and
\[ \ci\isom D\Ext^2_{\langle M\rangle}(M^\ss,M^\ss).\]
Moreover, $B/\ci$ is the basic algebra attached to $A(M)$.
\end{thm}
\begin{proof}We combine Proposition~\ref{P:SpeciesWithRel} with 
Theorem~\ref{thm:generate_ideal}. As the maximal length of a path in the species is $2$, we have $\rr^3=0$, so Corollary~\ref{C:LongRelations} applies.

The algebra $A(M)$ and $B/\ci$ are Morita-equvialent and $B/\ci$ is basic. 
\end{proof}

In keeping with the notation of
\cite{hodge3} and  \cite{huber-wuestholz}, we write
\begin{align*}
V_\sing:\onemot&\to\modules{\Q},\\
V_\dR:\onemot&\to\modules{\Qbar}
\end{align*}
instead of $\hsing|_{\onemot}$ and $\hdR|_{\onemot}$ for the singular and de Rham realisation, respectively. Moreover, there is a functorial isomorphism
\[ V_\sing\tensor_\Q\C\isom V_\dR\tensor_\Qbar\C.\]
We refer to \cite[Section~8.1]{huber-wuestholz} for details. 
For $M\in\onemot$, we denote $\Per\langle M\rangle$ its period space, see Definition~\ref{defn:period-space}.
The main result of \cite{huber-wuestholz} shows the period conjecture in this case.

\begin{thm}[Huber--Wüstholz, {\cite[Theorem~13.3]{huber-wuestholz}}]\label{thm:HW}
The period conjecture holds for all objects of $\onemot$, i.e.
\[ \dim_\Q A(M)=\dim_\Qbar\Per\langle M\rangle\]
for all $M\in\onemot$.
\end{thm}

\subsection{The category of \texorpdfstring{$1$}{1}-motives}\label{ssec:cat_1}
We now  describe the category $\onemot$ in more detail. Objects are complexes of abelian group schemes over $\Qbar$ of the form
 \[ M=[L\xrightarrow{f} G]\]
where $G$ is a semi-abelian variety over $\Qbar$ (i.e., an extension $0\to T\to G\to A\to 0$ of an abelian variety $A$ by a torus $T\isom\Gm^r\isom (\Qbar^*)^r$), $L\isom\Z^s$ a finitely generated free abelian group, and $f$ a group homomorphism. Morphisms are morphisms of complexes of abelian group schemes tensored with $\Q$. So they are given by compatible pairs $(\phi_L,\phi_G)$ where $\phi_L$ is a homomorphism of abelian groups
and $\phi_G$ a morphism of algebraic groups.

There are three types of simple objects:
\begin{enumerate}
\item $\Q(0):=[\Z\to 0]$, pure of weight $0$ and with field of endomorphisms $D(\Q(0))=\Q$, singular realisation $V_\sing(\Q(0))=\Q$ and multiplicity 
\[ m_0=\dim_{\Q}\Vsing{\Q(0)}=1;\]
\item $\Q(1):=[0\to\Gm]$, pure of weight $-2$ and with field of endomorphisms $D(\Q(1))=\Q$, singular realisation $V_\sing(\Q(1))=\Q$ and multiplicity 
\[ m_1=\dim_{\Q}\Vsing{\Q(1)}=1;\]
\item $\ul{A}:=[0\to A]$ for a simple abelian variety $A$ of dimension $g_A$, pure of weight $-1$ with a skew field $D(\ul{A})$  of endomorphisms of $\Q$-dimension $d_A$, singular realisation $V_\sing(\ul{A})=H_1(A,\Q)$ of $\Q$-dimension $2g_A$ and hence multiplicity 
\[ m_A=\dim_{D(\ul{A})}\Vsing{\ul{A}}=2g_A/d_A.\]
\end{enumerate}
\begin{rem}
We have  $D(\ul{A})=\End(A)\tensor\Q$. A common notation for this algebra in the theory of abelian varieties is
$E^0(A)$. In \cite{huber-wuestholz}, its dimension was denoted $e_A$ instead of $d_A$.
\end{rem}

\begin{cor}
Let $\csmult_M=(D_i^{m_i},D\Ext^1_{\langle M\rangle}(S_j^{m_j},S_i^{m_i}))$ be the species with multiplicities defined by $\langle M\rangle$ with multiplicities as above. Then
\[ A(M)\isom \Q\csmult_A/\ci\]
and
\[ \ci\isom D\Ext^2_{\langle M\rangle}\left(\bigoplus_i S_i^{m_i},\bigoplus_iS_i^{m_i}\right) \ .\]
\end{cor}
\begin{proof}This is Corollary~\ref{C:LongRelations} for the non-basic $A(M)$
and the hereditary $H=\Q\csmult_M$.
\end{proof}

\begin{ex}\label{ex:kummer}
For every $\alpha\in\Qbar^*$, we have the \emph{Kummer motive}
$M(\alpha)=[\Z\xrightarrow{1\mapsto \alpha}\G_m]$. It is an extension
\[ 0\to \Q(1)\to M(\alpha)\to \Q(0)\to 0.\]
This construction defines the \emph{Kummer map}
\[ \Qbar^*\to \Ext^1_{\onemot}(\Q(0),\Q(1)).\]
It is an isomorphism after passing from the abelian group $\Qbar^*$ to the $\Q$-vector space $\Qbar^*\tensor_\Z\Q$.

The species of $\langle M(\alpha)\rangle$ consists of $D_0=D_1=\Q$ and ${}_0E_1=D\Q[\alpha]$ (the dual of the $\Q$-subspace
 of $\Qbar^*\tensor_\Z\Q$ generated by $\alpha\tensor 1$). The space $\Q[\alpha]$ is zero if and only if $\alpha$ is torsion, i.e., a root of unity. 
The path algebra of the species is the path algebra of the quiver with two vertices and a single edge (or no edge if $\alpha$ is a root of unity). It has dimension $3$ (or $2$, if $\alpha$ is a root of unity). There are no relations because the maximal length of a path is $1$.

 The period space of $M(\alpha)$ is generated by $1$ (the period of $\Q(0)$),
$2\pi i$ (the period of $\Q(1)$) and $\log(\alpha)$.  Theorem~\ref{thm:HW} says that $\Per\langle M(\alpha)\rangle$ also has dimension $3$ unless $\alpha$ is a root of unity. This means
that $2\pi i$ and $\log(\alpha)$ are transcendental (Lindemann's Theorem) and
$\Qbar$-linearly independent (Gelfond-Schneider).
\end{ex}

\begin{ex} Let $\alpha_1,\dots,\alpha_n\in\Qbar^*$. We consider
\[ M=M(\alpha_1)\oplus\dots\oplus M(\alpha_n)\]
where $M(\alpha_i)$ is the Kummer motive of the preceding example.
The species of $M$ has two vertices $0,1$ with simple algebras
$D_0=D_1=\Q$ and 
\[ D({}_0E_1)=\langle \alpha_1\tensor 1,\dots,\alpha_n\tensor 1\rangle_\Q\subset \Qbar^*\tensor_\Z\Qbar.\]
Its $\Q$-dimension is equal to the rank of the subgroup of $\Qbar^*$ generated multiplicatively by $\alpha_1,\dots,\alpha_n$. Hence the $\Q$-dimension of the path algebra is 
\[ \delta=2+\rk\langle\alpha_1,\dots,\alpha_n\rangle.\]
Translated to periods this means that the $\Qbar$-vector space 
\[ \langle 1,2\pi i,\log(\alpha_1),\dots,\log(\alpha_n)\rangle_\Qbar\subset\C\]
has dimension $\delta$. This is Baker's Theorem.
\end{ex}

\subsection{Dimension formulas for \texorpdfstring{$1$}{1}-motives}\label{ssec:dim_1}
Given $M\in\onemot$, the quiver of the species of $\langle M\rangle$ is a subquiver of the following quiver:
%%%%%%%%%%%%%%%%%%%%%%%%%%%%%%%%%%%
\begin{equation*}\begin{tikzpicture}[description/.style={fill=white,inner sep=2pt}]
   \matrix (m) [matrix of math nodes, row sep=3em,
                 column sep=2.5em, text height=1.5ex, text depth=0.25ex,
                 inner sep=2pt, nodes={inner xsep=0.3333em, inner
ysep=0.3333em}] at (0, 0)
    {  &&D_1 \\
        &&D_2   \\
       D_0&&\ldots && D_\infty \\
       &&D_{n} \\
    };

 \path[->] ($(m-3-1.east)$)  edge  node [fill=white, scale=0.75] [midway] {$E_{01}$} ($(m-1-3.west)$) ;
 \path[->] ($(m-3-1.east) $) edge  node [fill=white, scale=0.75] [midway] {$E_{02}$} ($(m-2-3.west)$) ;
 \path[->] ($(m-3-1.east) $) edge  node [fill=white, scale=0.75] [midway] {$E_{0i}$} ($(m-3-3.west)$) ;
  \path[->] ($(m-3-1.east) $) edge  node [fill=white, scale=0.75] [midway] {$E_{0n}$} ($(m-4-3.west)$) ;

 \path[->] ($(m-1-3.east)$)  edge  node [fill=white, scale=0.75] [midway] {$E_{1\infty}$} ($(m-3-5.west)+(0, 2mm)$) ;
 \path[->] ($(m-2-3.east) $) edge  node [fill=white, scale=0.75] [midway] {$E_{2\infty}$} ($(m-3-5.west)+(0, 1mm)$) ;
 \path[->] ($(m-3-3.east) $) edge  node [fill=white, scale=0.75] [midway] {$E_{i\infty}$} ($(m-3-5.west)$) ;
  \path[->] ($(m-4-3.east) $) edge  node [fill=white, scale=0.75] [midway] {$E_{n\infty}$} ($(m-3-5.west)+(0, -1mm)$) ;

 \draw[ ->] ($(m-3-1.east) + (0mm,0mm)$) .. controls +(2.5mm,-35mm) and
+(-2.5mm,-35mm) .. node [fill=white, scale=0.75] [midway] {$E_{0\infty}$} ($(m-3-5.west) + (0mm,-2mm)$);
 \end{tikzpicture}\end{equation*}

As already discussed in \cite[Chapter~16]{huber-wuestholz}, the period space of
$M$ decomposes into contributions from different constituents. Under the period conjecture, Theorem~\ref{thm:HW}, this is precisely the decomposition of 
Theorem~\ref{thm:dim_formula}. The dictionary between \cite[Chapter~16]{huber-wuestholz} and Definition~\ref{defn:delta} is as follows:\\

\begin{center}
\begin{tabular}{cccl}
$\delta_\alg(M)$&	$=$	&$\delta(\Q(0))$ &algebraic periods\\[0.3ex]
$\delta_2(M)$	&$=$	& $\sum_Am^2_A\delta(\ul{A})$& abelian periods of the second kind\\[0.3ex]
$\delta_\tate(M)$&$=$		& $\delta(\Q(1))$ & Tate periods\\[0.3ex]
$\delta_3(M)$	&$=$	& $\sum_Am_A\delta(\ul{A}\Q(1))$ & 3rd kind wrt closed paths\\[0.3ex]
$\delta_\inc(M)$&$=$	& $\sum_Am_A\delta(\Q(0)\ul{A})$ &2nd kind wrt non-cl. paths\\[0.3ex]
$\delta_\mix(M)$&$=$	& $\delta(\Q(0)\Q(1))$ &3rd kind wrt non-cl. paths
\end{tabular}
\end{center}

\noindent where all sums are over the simple abelian subquotients of $M$, i.e., the vertices of weight $-1$ in the quiver. The names refer to the situation in which these periods appear ``in nature''. E.g., the periods of $\Q(0)$ are algebraic numbers; the periods of $\ul{A}$ can be represented as integrals of algebraic differential form of the second kind (all residues vanish) over closed paths on a smooth projective curve over $\Qbar$. For a complete discussion see \cite[Chapter~14]{huber-wuestholz}.

\begin{thm}\label{thm:main1}
Let $M\in\onemot$ and $\ul{A_1},\dots,\ul{A_n}$ be the simple abelian subquotients of $M$  (without multiplicities). Then
\[ \dim_\Qbar\Per\langle M\rangle=\delta_\alg(M)+\delta_2(M) +\delta_\tate(M)+\delta_3(M)+\delta_\inc(M)+\delta_\mix(M)\]
where $\delta_\alg(M)$ and $\delta_\tate(M)$ take the values $0$ or $1$, depending on the vanishing or non-vanishing of the lattice part $L$ or the torus part $T$ of $M$, respectively. Moreover,
\begin{align*}
\delta_2(M)&=\sum_{i=1}^n  4g_{A_i}^2/d_{A_i}\\
\delta_3(M)&=\sum_{i=1}^n  2g_{A_i}\dim_{D(\ul{A_i})} \Ext^1_{\langle M\rangle}(\ul{A_i},\Q(1))\\
\delta_\inc(M)&=\sum_{i=1}^n 2g_{A_i}\dim_{D(\ul{A_i})} \Ext^1_{\langle M\rangle}(\Q(0),\ul{A_i})\\
\delta_\mix(M)&=\dim_\Q\Ext^1_{\langle M\rangle}(\Q(0),\Q(1))-\dim_\Q\Ext^2_{\langle M\rangle}(\Q(0),\Q(1))\\
      & +\sum_{i=1}^n d_{A_i}\left(\dim_{D(A_i)}\Ext^1_{\langle M\rangle}(\Q(0)\ul{A_i})\cdot\dim_{D(A_i)}\Ext^1_{\langle M\rangle}(\ul{A_i},\Q(1))\right)
\end{align*}
with all $\Ext$-groups taken in $\langle M\rangle$.
We have
\begin{equation}\label{eq:trivial_bound} \delta_\mix(M)\leq lt\end{equation}
where $l$ is the rank of the lattice part of $M$ and $t$ is the dimension of 
the torus part of $M$.
\end{thm}
\begin{proof}By Theorem~\ref{thm:HW}, we have to
compute the dimension of the Nori algebra $A(M)$. It is not basic. By Lemma~\ref{lem:mult} and the computation at the beginning of Section~\ref{ssec:cat_1} the vertex
corresponding to the simple abelian variety $A_i$  appears with muliplicity $m_{A_i}=2g_{A_i}/d_{A_i}$. The vertices corresponding to $\Q(0)$ and $\Q(1)$ appear with multiplicity $1$. 

We now apply Corollary~\ref{c: lengthatmost2}  (the case of maximal path length at most $2$) in terms of the dictionary explained above.

Note that 
\begin{align*}
m_{A_i}^2\delta(\ul{A_i})&=m_{A_i}^2d_{A_i}= \frac{(2g_{A_i})^2}{d_{A_i}},\\
m_{A_i}\delta(\Q(0)\ul{A_i})&=m_{A_i}\dim_{\Q}\Ext^1_{\langle M\rangle}(\Q(0),\ul{A_i})\\
   &= \frac{2g_{A_i}}{d_{A_i}}d_{A_i}\dim_{D(A_i)}\Ext^1_{\langle M\rangle}(\Q(0),\ul{A_i}),\\
m_{A_i}\delta(\ul{A_i}\Q(1))&=m_{A_i}\dim_{\Q}\Ext^1_{\langle M\rangle}(\ul{A_i},\Q(1))\\
&=
    \frac{2g_{A_i}}{d_{A_i}}d_{A_i}\dim_{D(A_i)}\Ext^1_{\langle M\rangle}(\ul{A_i},\Q(1))\ .
\end{align*} 
In computing the contribution of the path $\Q(0)\to \ul{A_i}\to\Q(1)$ to
$\delta_\mix(M)$ we use
\begin{align*}
\dim_\Q D\Ext^1_{\langle M\rangle}&(\Q(0),\ul{A_i})\tensor_{D(\ul{A_i})}D\Ext^1_{\langle M\rangle}(\ul{A_i},\Q(1)) \\
 &=\frac{1}{d_{A_i}}\dim_\Q\left(D\Ext^1_{\langle M\rangle}(\Q(0),\ul{A_i})\tensor_{\Q}D\Ext^1_{\langle M\rangle}(\ul{A_i},\Q(1))\right) \\
 &= \frac{1}{d_{A_i}} \dim_\Q\Ext^1_{\langle M\rangle}(\Q(0),\ul{A_i})\cdot\dim_\Q\Ext^1_{\langle M\rangle}(\ul{A_i},\Q(1)) \\
&=d_{A_i}\dim_{D(\ul{A_i})}\Ext^1_{\langle M\rangle}(\Q(0),\ul{A_i})\cdot\dim_{D(\ul{A_i})}\Ext^1_{\langle M\rangle}(\ul{A_i},\Q(1))\ . 
\end{align*}

For the upper bound, we apply Lemma~\ref{L:trivial_bound} to $A=A(M)$, the module $V=V_\sing(M)$ and the projectors corresponding to the simple objects 
$\Q(0)$ and $\Q(1)$:
\[ \delta_\mix(M)=\dim_\Q e_{\Q(0)} A(M)e_{\Q(1)}\leq \dim_\Q e_{\Q(0)}V\cdot\dim_\Q e_{\Q(1)}V=lt.\]
\end{proof}

\begin{ex}We come back to the Kummer motive introduced in Example~\ref{ex:kummer}. Assume that $\alpha\in\Qbar^*$ is not a root of unity. Then the non-vanishing contributions are
\[ \delta_\alg(M(\alpha))=\delta_\tate(M(\alpha))=\delta_\mix(M(\alpha))=1\]
for a total of
\[ \delta(M(\alpha))=3\]
as before. Indeed, the period $\log(\alpha)=\int_1^\alpha d t/t$ is an integral of the third kind (non-trivial residues) over a non-closed path. As such it contributes to $\delta_\mix(M(\alpha))$.
\end{ex}

A non-hereditary example will be discussed in detail in Section~\ref{ssec:ex_nh}.

\begin{rem}We explain the relation of Theorem~\ref{thm:main1} to the dimension formula in \cite{huber-wuestholz}. \begin{enumerate}
\item The main new insight is the formula for $\delta_\mix(M)$, the most complicated contribution. 
\item Nesa already established in \cite[Chapter~3]{nesa}, see also \cite[Section~3]{nesa-ext} that the ``ranks'' appearing in \cite[Notation~15.2]{huber-wuestholz} have a more conceptual interpretation as the $D(\ul{A})$-dimension of $\Ext^1$ in $\langle M\rangle$. We recover his result.
\item
If $M$ is the sum of a saturated motive and  a Baker motive in the terminology of \cite[Chapter~15]{huber-wuestholz}, 
the category $\langle M\rangle$ is hereditary by \cite[Lemma~5.7]{huber-structure}, so the $\Ext^2$-contribution vanishes. 
We reproduce the formula in the saturated case, \cite[Theorem~15.3]{huber-wuestholz}. This includes cases where not all types of vertices appear in the quiver.
The connection between the hereditarity condition and the saturated case (and the description by the tensor algebra) is also shown by Hörmann in \cite{hoermann-notiz}.
\item For general $M$, the hereditary closure of $\langle M\rangle$ (see Chapter~\ref{sec3}) can be used in order to find an upper bound for the dimension of the period space. The existence of a hereditary closure is Theorem~\ref{thm:sat_exists}. It was also shown by hand in \cite[Theorem~15.3.~(2)]{huber-wuestholz}. The dimension estimate of \cite[Theorem~15.3]{huber-wuestholz} is  the same estimate as above. It omits the correction by $\dim_\Q\Ext^2(\Q(0),\Q(1))$.
\item The upper bound for $\delta_\mix(M)$ in Theorem~\ref{thm:main1} \eqref{eq:trivial_bound} is proved directly in \cite[Corollary~17.2]{huber-wuestholz}.
Chapter~17 of \cite{huber-wuestholz} gives an algorithm to determine the
dimension of the contribution $\delta_\mix(M)$  in the general case. 
Our new description in terms of $\Ext^2$ provides a clean conceptual explanation. \end{enumerate}
\end{rem}

\subsection{A non-hereditary example}\label{ssec:ex_nh}

We use the construction of
\cite[Section~11.3]{huber-wuestholz}. Let $E$ be a CM-elliptic curve, so that
$F:=D_E$ is a quadratic imaginary field extension of $\Q$. Let 
$G$ be an extension of $E$ by $\Gm$ which is not split, even up to isogeny. It corresponds to a non-torsion element $\chi$ of $E^\vee(\Qbar)$ (where $E^\vee$ is the dual abelian variety, isomorphic to $E$ in the $1$-dimensional case). Let
$\alpha\in G(\Qbar)$ be such that the image $\bar{\alpha}$ in $E(\Qbar)$ is not torsion.
Let $M=[\Z\xrightarrow{1\mapsto \alpha}G]$. Its periods are explicitly described in \cite[Chapter~18]{huber-wuestholz} as elliptic integrals.

We apply our machine. 
The simple objects of $\langle M\rangle$ are $\Q(0)$, $\Q(1)$ and $\ul{E}$. 
The species has $D_0=D_1=\Q$ and $F:=D_{\ul{E}}$ imaginary quadratic over $\Q$. 
All multiplicities are equal to $1$, so
the algebra $A(M)$ is basic. By \cite[Prop.~3.8, Prop.~3.11]{nesa-ext} 
both
\begin{align*}
{}_{\ul{E}}E_1&=D\Ext^1(\ul{E},\Q(1))\\
{}_0E_{\ul{E}}&=D\Ext^1(\Q(0),\ul{E})
\end{align*}
 are of dimension $1$ over $F$, so of dimension $2$ over $\Q$. 
Let 
\[ \tau=\dim_\Q{}_1E_0=\dim_\Q\Ext^1_{\langle M\rangle}(\Q(0),\Q(1)).\]
(We will show that $\tau=0$ below.)
By Theorem~\ref{thm:main1}
\begin{align*}
\delta_\alg(M)&=\delta_\tate(M)=1,\\
 \delta_2(M)&=\delta_3(M)=\delta_\inc(M)=2,\\
\delta_\mix(M)&=2+\tau-\dim_\Q\Ext^2_{\langle M\rangle}(\Q(0),\Q(1)).
\end{align*}
We have $l=t=1$ in $M$, hence we also have
\[ \delta_\mix(M)\leq 1.\]
Together this implies that the category is not hereditary. 

We can be more precise. Let $\Sh_M$ be the species of $\langle M\rangle$. We have 
\[ A(M)\isom \Q\Sh_M/\Ih\]
for the admissible ideal $\Ih\isom D\Ext^2_{\langle M\rangle}(\Q(0),\Q(1))$.
The base change of $\Q\Sh_M$ to  $\Qbar$ is the path algebra of the quiver
 considered in Example~\ref{ex:dimension}. The motive $M$ corresponds to the representation $V$ considered there. The ideal of relations is admissible. By Example~\ref{ex:dimension}, we have $\tau=0$ and $\dim\Ih=1$. This implies
\[ \dim_\Q\Ext^2_{\langle M\rangle}(\Q(0),\Q(1))=1\]
and
\[ \delta_\mix(M)=1.\]
\begin{rem}
This is the same number obtained by Huber--Wüstholz by hand (\cite[Section~11.3]{huber-wuestholz} or their algorithm for $\delta_\mix(M)$ (\cite[Example~17.16]{huber-wuestholz}).
\end{rem}
We describe the non-trivial $2$-extension. The quiver representation
$V$ in Example~\ref{ex:dimension}  (for $\tau=0$) corresponds to the projective left module $(\Qbar Q/\Ih) e_1$. This makes $M$ the projective cover of $\Q(0)$
in $\langle M\rangle$. The projective cover of $\ul{E}$ corresponds to the projective module $(\Qbar Q/\Ih)(e_2+e_3)$ in Example~\ref{ex:dimension}. It is the motive $\ul{G^\sat}$ where
\[ G^\sat=G\times_E\iota^*G\]
with $F=\Q(\iota)$. Here $\iota^*G$ is the pull-back of $G$ along $\iota:E\to E$. Note that it is isomorphic to $G$ (via $\iota$) as a semi-abelian variety, but not as an extension of $E$.
The character group of the torus part of $G^\sat$ is $X(T^\sat)=F\chi\subset E^\vee(\Qbar)_\Q$.
The motive $\Q(1)$ is projective itself. It corresponds to the idempotent $e_4$ in Example~\ref{ex:dimension}.

Hence, the complex
\[ \Q(1)\xrightarrow{\delta} \ul{G^\sat}\to M\]
is a projective resolution of $\Q(0)$. This implies
\[ \Ext^2_{\langle M\rangle}(\Q(0),\Q(1))=\Hom_{\langle M\rangle}(\Q(1),\Q(1))/\delta^*\Hom(\ul{G^\sat},\Q(1))\isom\Q.\]
The non-trivial $2$-extension is
\[ 0\to \Q(1)\to \ul{G}\xrightarrow{f} [\Z\xrightarrow{1\mapsto\bar{\alpha}} E]\to \Q(0)\to 0\]
where 
\[ f:[0\to G]\twoheadrightarrow [0\to E]\xrightarrow{\iota}[0\to E]\hookrightarrow [\Z\to E]\]
uses multiplication by $\iota\in F\smallsetminus\Q$.

By Theorem~\ref{thm:sat_exists}, there is a minimal full hereditary category
closed under subquotients
$\langle M\rangle^\sat\subset\onemot$ containing $M$. We want to describe it in our case. The category is equivalent to
$\modules{\Q\Sh_M}$.
By Corollary~\ref{cor:Msat} it is generated by an object $M^\sat$, corresponding to $\Q\Sh_M$ as a module over itself. We already have computed the images of the idempotents modulo $\Ih$. This gives
\[ M^\sat=\Q(1)\oplus \ul{G^\sat} \oplus [\Z\xrightarrow{\phi} G^\sat]\]
where $\phi(1)$ is induced by $(\alpha,\iota^{-1}(\alpha))\in G\times \iota^*(G)$.
(Note that they have the same image in $E$ under the structure maps of
$G$ and $\iota^*(G)$.)
Actually, the category $\langle M\rangle^\sat$ is generated by $[\Z\to G^\sat]$,
the projective cover of $\Q(0)$ in the saturation, since $\Q(1)$ and $G^\sat$ are subobjects.

\begin{rem} This does not agree with the construction of an object $M_\sat$ in \cite[Chapter~15]{huber-wuestholz}. We will denote it
$M^{\HW}$ for distinction. Let us review $M^{\HW}$ in our example. It is of the form
\[ M^{\HW}=[ \Oh_F\xrightarrow{\psi} G^\sat]\]
with the same $G^\sat$ as above and $\psi$ the $\Oh_F$-linear map mapping
$1$ to $(\alpha,\beta)$ for arbitrary $\beta\in \iota^*(G)$ in the preimage of
$\bar{\alpha}$. The category $\langle M^{\HW}\rangle$ is indeed hereditary and contains $\langle M\rangle$. Hence
\[ \langle M\rangle^\sat\subset \langle M^{\HW}\rangle.\]
The difference $\beta-\iota^{-1}(\alpha)$ defines an element in
$\Ext^1_{\langle M^{\HW}\rangle}(\Q(0),\Q(1))$.
It is not an extension in $\langle M\rangle^\sat$ because 
\[ \Ext^1_{\langle M\rangle^\sat}(\Q(0),\Q(1))=\Ext^1_{\langle M\rangle}(\Q(0),\Q(1))=0.\]
Hence $\langle M\rangle^{\sat}\subsetneq \langle M^{\HW}\rangle$ if $\beta\neq\iota^{-1}(\alpha)$.
\end{rem}

%% file: kapitel6.tex
\section{Mixed Tate Motives}\label{sec:MTM}
We apply our methods to periods of Mixed Tate Motives over a fixed number field $F$. This allows us to recover the dimension estimate of Terasoma in \cite{terasoma} and  Deligne and Goncharov in \cite{deligne-goncharov}, see Section~\ref{ssec:MTMZ}.

\subsection{The category}

Levine constructed in \cite{LeMTM}, see also \cite{HuberKahn}, a $\Q$-linear abelian category $\MTM_F$ of Mixed Tate Motives over $F$ inside Voevodsky's triangulated category of geometric motives $\DMgm(F)$. Its simple
objects are $\Q(i)$ for $i\in \Z$ with endomorphism ring $\Q$.  The construction relies on the identification 
%\begin{multline*}
% \Ext^1_{\MTM_F}(\Q(n),\Q(n+m))\isom \Ext^1_{\MTM_F}(\Q,\Q(m)) \\
%\isom \Hom_{\DMgm(F)}(\Q,\Q(m)[1])\isom K_{2m-1}(F)^{(m)}_\Q\isom K_{2m-1}(F)_\Q.
%\end{multline*}
\begin{multline*}
 \Ext^i_{\MTM_F}(\Q(n),\Q(n+m)):= \Hom_{\DMgm(F)}(\Q(n),\Q(n+m)[i]) \\
\isom \Hom_{\DMgm(F)}(\Q,\Q(m)[i])\isom K_{2m-i}(F)^{(m)}_\Q\isom 
\begin{cases}
\Q &i=0,m=0,\\
K_{2m-1}(F)_\Q &i=1,\\
0& \text{else.}
\end{cases}
\end{multline*}

By \cite[Proposition~12.2]{borel}  they are finite dimensional for $m\neq 1$. 
They vanish for $m\leq 0$. This vanishing induces a weight filtration.  To be consistent with the general setting, $\Q(n)$ is given weight $-2n$. The category is hereditary, again by Borel's computation. The assumptions of Section~\ref{sec3} are satisfied.

Note, moreover, that
\[ \Ext^1_{\MTM_F}(\Q(0),\Q(1))\isom F^\times_\Q:=F^\times\tensor_\Z\Q.\]

\begin{lemma}\label{lem:construct_tate}
For any choice of a finite dimensional subvector space
$V\subset F^\times_\Q$ there is a maximal full abelian subcategory $\MTM_{F,V}$ of $\MTM_F$ with the same simple objects and such that
\[ \Ext^1_{\MTM_{F,V}}(\Q(0),\Q(m))=\begin{cases} \Ext^1_{\MTM_F}(\Q(0),\Q(m)) &m\neq 1,\\
V&m=1.\end{cases}
\]
Its objects are the $X\in \MTM_F$ whose subquotients of length  $2$ (i.e., $1$-extensions of simple objects) are of the type prescribed by the above $\Ext^1$'s.
The category is hereditary.
\end{lemma}
\begin{proof}
We apply Corollary~\ref{cor:finite_ext}. 
\end{proof}
The same result was obtained by Deligne--Goncharov with Tannakian methods, see \cite[Proposition~1.9]{deligne-goncharov}.

\begin{defn}[{\cite[Definition~1.6]{deligne-goncharov}}]For every finite set $S$ of prime ideals, we define the category
$\MTM_{\Oh_F[S^{-1}]}$ of \emph{Mixed Tate Motives over $\Oh_F[S^{-1}]$} as the category in the lemma for the choice $V=\Oh_F[S^{-1}]^\times_\Q\subset F^\times_\Q$.
\end{defn}
The group $\Oh_F[S^{-1}]^\times$ is finitely generated by Dirichlet's Unit Theorem. Hence,
 all $\Hom$- and $\Ext$-groups are finite dimensional in these categories.

The singular realisation is faithful and exact on this category. It maps $\Q(n)$ to $\Q$.

\begin{defn}Let $\MTM^{[0,n]}_{\Oh_F[S^{-1}]}$ be the full subcategory closed under extensions containing $\Q(0),\dots,\Q(n)$. 
Let
\[ B_n=B_n(\Oh_F[S^{-1}])=\End\left(H_\sing|_{\MTM^{[0,n]}_{\Oh_F[S^{-1}]}}\right)\]
be the corresponding Nori algebra.
\end{defn}
Note that $\MTM^{[0,n]}_{\Oh_F[S^{-1}]}$ is still hereditary since it is extension closed in a hereditary category.

 By Lemma~\ref{lem:mult}, each $\Q(n)$ has multiplicity $1$ and endomorphism algebra $\Q$. Recall the notion of species and its path algebra from Section~\ref{sec:species}.

\begin{prop}The category $\MTM^{[0,n]}_{\Oh_F[S^{-1}]}$ is equivalent to the category $\modules{B_n(\Oh_F[S^{-1}])}$.

The algebra $B_n(\Oh_F[S^{-1}])$ is basic and hereditary. It is the path algebra of the species $\Sh_n$ with 
\begin{itemize}
\item vertices $0,\dots,n$, 
\item division ring $D_i=\Q$ for $i=0,\dots, n$, and
\item ${}_jE_i\isom D\Ext^1_{\MTM(\Oh_F[S^{-1}])}(\Q(j),\Q(i))$. 
\end{itemize}
In other words, $B_n(\Oh_F[S^{-1}])$ is the path algebra of the quiver with vertices $0,\dots,n$ and $e_{i-j}$ many edges from $j$ to $i$ where
\[ e_s:=\dim_\Q\Ext^1_{\MTM(\Oh_F[S^{-1}])}(\Q(0),\Q(s)).\]
\end{prop}
\begin{proof}The first statement is due to Nori, see also \cite[Section~7.3]{period-buch}. 

The algebra is basic because the multiplicities are $1$. It is hereditary because $\MTM^{[0,n]}_{\Oh_F[S^{-1}]}$ is. It is isomorphic to the path algebra of the species by the hereditary case of Proposition~\ref{P:SpeciesWithRel}. 

The identification of the path algebra of the species and the path algebra of a quiver with multiple edges is Example~\ref{ex:Q}.
\end{proof}
Analogously, we define $\MTM^{[a,b]}_{\Oh_F[S^{-1}]}$ for $a\leq b$ as the full abelian subcategory closed under extensions generated by $\Q(a), \dots,\Q(b)$. The category is equivalent to $\MTM^{[0,b-a]}_{\Oh_F[S^{-1}]}$ by tensoring with $\Q(-a)$. 
For $0\leq a\leq b\leq n$ it is contained in $\MTM^{[0,n]}_{\Oh_F[S^{-1}]}$. 
This induces an inclusion of their period spaces.

The following observation is our formulation of a key insight of \cite{deligne-goncharov}:
\begin{cor}\label{cor:MT_dim} 
Abbreviating $\MTM=\MTM_{\Oh_F[S^{-1}]}$ and
$B_n=B_n(\Oh_F[S^{-1}])$, we have
\begin{align}
\label{eq:pc1} \dim_{F} \Per(\MTM^{[0,n]})&\leq \dim_\Q B_n,\\
\label{eq:pc2}\dim_{F} \Per(\MTM^{[0,n]})/\Per(\MTM^{[0,n-1]})&\leq \dim_\Q B_n-\dim_\Q B_{n-1},\\
\label{eq:pc3}\dim_{F}\Per(\MTM^{[0,n]})/(\Per(\MTM^{[0,n-1]}+&\Per(\MTM^{[1,n]}))\\
&\leq \dim B_n-2\dim B_{n-1}+\dim B_{n-2}.
\end{align}
The Period Conjecture for Mixed Tate Motives over $\Oh_F[S^{-1}]$ is equivalent to equality for all $n$ in one of the formulas.
\end{cor}
\begin{proof}In the language of \cite{huber_galois}:
The space of periods $\Per(\MTM^{[0,n]})$ is a quotient of the space of formal periods $\Pertilde(\MTM^{[0,n]})$. The latter has the same dimension as $B_n$. This proves \eqref{eq:pc1}. Equality in \eqref{eq:pc1} is equivalent to the Period Conjecture by
\cite[Proposition~5.7]{huber_galois}.

The induced surjection
\[ \Pertilde(\MTM^{[0,n]})/\Pertilde(\MTM^{[0,n-1]})\twoheadrightarrow \Per(\MTM^{[0,n]})/\Per(\MTM^{[0,n-1]})\]
implies the dimension estimate in \eqref{eq:pc2}.

Finally, note that $\MTM^{[0,n-1]}\cap \MTM^{[1,n]}=\MTM^{[1,n-1]}$ and hence
\begin{multline*}
 \Pertilde(\MTM^{[0,n-1]})+\Pertilde(\MTM^{[1,n]})\\
\isom
    \Pertilde(\MTM^{[0,n-1]})\oplus\Pertilde(\MTM^{[1,n]})/\Pertilde(\MTM^{[1,n-1]}).
\end{multline*}
Together with the surjection
\begin{multline*}
 \Pertilde(\MTM^{[0,n]})/(\Pertilde(\MTM^{[0,n-1]}+\Pertilde(\MTM^{[1,n]}))\\
\twoheadrightarrow
\Per(\MTM^{[0,n]})/(\Per(\MTM^{[0,n-1]}+\Per(\MTM^{[1,n]}))
\end{multline*}
this implies the dimension estimate \eqref{eq:pc3}.

Equality for all $n$ in \eqref{eq:pc3} or \eqref{eq:pc2} implies inductively
equality for all $n$ in \eqref{eq:pc1} and hence the Period Conjecture.

\end{proof}

\begin{rem}
There is an extensive literature on periods of Mixed Tate Motives. We do not attempt to survey it, but see for example \cite{burgos-fresan} by Burgos Gil and Fres\'an or \cite[Chapter~15]{period-buch}.
 Terasoma  in \cite{terasoma} and independently Deligne and Gocharov in \cite{deligne-goncharov} showed that the number $\zeta(n_1,\dots,n_m)$ appears as the period of a Mixed Tate Motive
 in $\MTM_{\Z}^{[0,N]}$ with $N=\sum_{i=1}^m n_i$ 
and deduce a bound on
the dimension of the space of multiple zeta values. We will recover their formula in 
Section~\ref{ssec:MTMZ}.

Brown \cite{brown-annals} even proved that all periods of Mixed Tate Motives over $\Z$ are (up to multiplication by powers of $2\pi i$) spanned by these multiple zeta values. 

It is not completely obvious how our path algebras translate to Goncharov's \emph{framed Mixed Tate Motives} in \cite{goncharov:galois_sym} and Brown's \emph{motivic zeta elements} in \cite{brown-annals}. We leave this to follow-up work. 
\end{rem}

\subsection{Dimension formulas}
We now turn to the computation of $\dim_\Q B_n(\Oh_F)$ and, hence, an upper bound for the dimension of spaces of periods of Mixed Tate Motives.
We put
\[ e_m=\dim_\Q\Ext^1_{\MTM_{\Oh_F[S^{-1}]}}(\Q(0),\Q(m)).\]
The dimensions are known by Borel's computation of algebraic $K$-groups in \cite[Proposition~12.2]{borel}. We have
\begin{align}\label{E:Borel}
e_m=\begin{cases} r_1+r_2+|S|-1&m=1\\
                     r_2&m\geq 2, \text{\ even}\\
                     r_1+r_2&m\geq 3, \text{\ odd}
               \end{cases}
\end{align}
where $r_1$ and $r_2$ denote the number of real and complex places of $F$, respectively, i.e.,  for $F=\Q[X]/(f)$, 
the irreducible polynomial $f$ has degree $r_1 + 2r_2$ with $r_1$ real roots and $2r_2$ roots in $\C \setminus \mathbb{R}$.

\begin{ex}\label{ex:field}
 Let $p$ be a prime. Then $F(p)=\Q(\sqrt[4]{p})$ with minimal polynomial $X^4-p$ has two real places (corresponding to the roots $\pm \sqrt[4]{p}$) and
one complex place (corresponding to $\pm i\sqrt[4]{p}$). If $p_1,\dots,p_n$ are different primes, then the compositum $F=F(p_1)\cdots F(p_n)$ of degree $4^n$ has
$r_1=2^{n}$ real places and 
$r_2=(4^n-2^n)/2=2^{n-1}(2^n-1)$ complex places. Already this simple example shows that $r_1$ and $r_2$ can be arbitrarily big. 
\end{ex}

\begin{ex}For
$n=0$, we have $B_0(\Oh_F[S^{-1}])=\Q$.
For $n=1$, we have 
\[ B_1(\Oh_F[S^{-1}])=\Q^2\bigoplus D\Ext^1_{\MTM_{\Oh_F[S^{-1}]}}(\Q(0),\Q(1))\isom \Q^2\oplus D\Oh_F[S^{-1}]^\times_\Q\]
and the dimension is $e_1+2$.
\end{ex}

Let $p_m$  be the number of paths from $0$ to $m$ in the quiver. 
We get a recursion formula for $p_m$.
\begin{lemma}\label{lem:compute-dim}
We have $p_0=1$ and for $m\geq 1$
\begin{align*}
 p_m&= \sum_{i=0}^{m-1}p_ie_{m-i}\\
    &=e_1p_{m-1}+(r_2+1)p_{m-2}+(1-|S|)p_{m-3}
\end{align*}
where we interpret $p_{-1}=p_{-2}=0$.
\end{lemma}
\begin{proof}We decompose a path from $0$ to $m$ into a path from $0$ to some $i$ between $0$ and $m$, composed with an edge from $i$ to $m$.

For the second equation, we consider $p_{2m}-p_{2m-2}$ and use the near $2$-periodicity of the $e_i$. This gives
\begin{multline*}
 p_{m}-p_{m-2}=p_{m-3}(e_3-e_1)+p_{m-2}e_2 +p_{m-1}e_1\\
=p_{m-3}(1-|S|)+p_{m-2}r_2 +p_{m-1}e_1.
\end{multline*}
Now collect terms.
\end{proof}

\begin{prop}\label{prop:sum_paths}
\[\dim_\Q B_n(\Oh_F[S^{-1}])=\sum_{0\leq i\leq j\leq n}p_{j-i}=\sum_{m=0}^{n} p_m(n-m+1)\]

\end{prop}
\begin{proof}
This is the dimension formula for the path algebra of a quiver with
$n+1$ vertices. Due to the symmetries of our quivers,
the number of paths from $i$ to $j$ is equal to the number of paths from $0$ to $j-i$. The distance $m$ appears $n-m+1$ times.
\end{proof}
\begin{cor}\label{cor:recursive1}
Recursively
\[\dim_\Q B_n(\Oh_F[S^{-1}])-\dim_\Q B_{n-1}(\Oh_F[S^{-1}])=\sum_{m=0}^{n}p_m.\]
\end{cor}
\begin{proof}
We take the difference:
\begin{multline*}
 \sum_{m=0}^{n}p_m(n-m+1)-\sum_{m=0}^{n-1}p_m(n-1-m+1)\\
        =p_{n}+\sum_{m=0}^{n-1}p_m(n-m +1 -n+m).
\end{multline*}
\end{proof}
\begin{cor}\label{cor:pathdim}
Let $F$ be a number field, $S$ a finite set of primes. We abbreviate
$\MTM=\MTM_{\Oh_F[S^{-1}]}$. Then
\begin{gather}
\label{eq:pc1a} \dim_{F} \Per(\MTM^{[0,n]})\leq \sum_{m=0}^{n} p_m(n-m+1),\\
\label{eq:pc2a}\dim_{F} \Per(\MTM^{[0,n]})/\Per(\MTM^{[0,n-1]})\leq \sum_{m=0}^n p_m\\
\label{eq:pc3a}\dim_{F}\Per(\MTM^{[0,n]})/(\Per(\MTM^{[0,n-1]}+\Per(\MTM^{[1,n]}))\leq p_n.
\end{gather}
Moreover,
\[ p_m= \sum_{i=0}^{m-1}p_ie_{m-i}
=p_{m-3}(1-|S|)+p_{m-2}r_2 +p_{m-1}e_1.\]

\end{cor}

\begin{proof}
We combine Corollary~\ref{cor:MT_dim} with Proposition~\ref{prop:sum_paths} for
\eqref{eq:pc1a} and with Corollary~\ref{cor:recursive1} for \eqref{eq:pc2a}.
For \eqref{eq:pc3a} we compute
\begin{multline*}
 \dim B_n-\dim B_{n-1}-\dim B_{n-1}+\dim B_{n-2}= \sum_{m=0}^np_m-\sum_{m=0}^{n-1}p_m=p_n
\end{multline*}
by Corollary~\ref{cor:recursive1}

The recursion formula is Lemma~\ref{lem:compute-dim}.
\end{proof}

%\begin{rem}\label{rem:MZV}
%In the case of Mixed Tate Motives over $\Z$, the period space in \eqref{eq:pc3a} is spanned by multiple zeta values of weight $n$, see Remark~\ref{rem:discussion_MZV} for details. 
%\end{rem}

%\subsection{Small examples}
\subsection{Mixed Tate motives over \texorpdfstring{$\Z$}{Z}}\label{ssec:MTMZ}

%\begin{ex}\label{ex:Z}
In this section, we spell out our results in the case $\MTM=\MTM_\Z$.
We have $r_1=1$, $r_2=0$, $|S|=0$ and hence \eqref{E:Borel} implies
\[ 0=e_1=e_2=e_4=\dots,\quad 1=e_3=e_5=\dots\]
Lemma \ref{lem:compute-dim} yields:
\begin{lemma}\label{lem:zagier}
In the case of Mixed Tate Motives over $\Z$, the dimension $p_m$ of the number of paths from $0$ to $m$ satisfies the recursion relation
\begin{equation}\label{zagier} p_{m+2}=p_m+p_{m-1}\end{equation}
with starting values $p_0=1$, $p_1=p_2=0$. 
\end{lemma}
The full sequence reads
\[ 1,0,0,1,0,1,1,1,2,2,3,4,5,7,9,\dots\]
From this we get the dimension of $B_n(\Z)$ and an upper bound for the space of periods $\dim_\Q\Per(\MTM_\Z^{[0,n]})$. 

This is very similar to the expected recursion relation for the dimension of spaces of multiple zeta values of weight $n$, but with different initial values. We thank Javier Fres\'an for pointing this out. We explain the connection.

\begin{defn}[{Zagier, \cite[p. 509]{zagier}}]\label{defn:recursion}
Let $d_n$ for $n\geq 0$ be given by $d_0=1$, $d_1=0$, $d_2=1$ and
\[ d_n=d_{n-2}+d_{n-3}.\]
\end{defn}
The sequence reads
\[1, 0, 1, 1, 1, 2, 2, 3, 4, 5, 7, 9, 12, 16,\dots\]

Recall that $B_n:=B_n(\Z)$ is the path algebra of the quiver $Q_n$ attached to $\MTM_\Z^{[0,n]}$. Let $P_i^{(n)}$ the the projective cover  of $\Q(i)$ in $\MTM^{[0,n]}$.
It corresponds to the $B_n$-module
$B_n\epsilon_i$ where $\epsilon_i$ is the idempotent for the vertex $i$ of the quiver. 
We have 
\[ B_n=\bigoplus_{i=0}^n P_i^{(n)}.\]
\begin{defn}We put
\[ B_n^\even=\bigoplus_{i=0}^{\lceil n/2\rceil} P_{2i}^{(n)},\quad B_n^\odd=\bigoplus_{i=0}^{\lceil n/2\rceil-1} P_{2i+1}^{(n)}.\]
\end{defn}
Note that
\begin{equation}\label{eq:sum} B_n=B_n^\even\oplus B_n^\odd\end{equation}
and
\begin{equation}\label{eq:odd} B_n^\odd\isom B_{n-1}^\even\end{equation}
as vector spaces.

\begin{prop}\label{prop:zagier_fits}
We have 
\[ \dim B_n^\even-\dim B_{n-1}^\even=d_n=p_{n+3}.\]
\end{prop}
\begin{proof}
Recall from the decomposition of the path algebra that $P_i^{(n)}$ has as basis the paths starting from $i$ in the quiver $Q_n$. Hence
\[ \dim P_i^{(n)}=\dim P_{0}^{(n-i)}=\sum_{j=0}^{n-i}p_j.\]
This implies
\begin{align*} \dim B_n^\even&=\sum_{i=0}^{\lceil n/2\rceil}\dim P_{2i}^{(n)}\\
            &=\sum_{i=0}^{\lceil n/2\rceil}\dim P_{0}^{(n-2i)}\\
            &=\dim P_0^{(n)}+\dim B_{n-2}^\even
\end{align*}
and hence
\[ \delta_n:=\dim B_n^\even-\dim B^{n-1}_\even=\dim P_0^{(n)}-\dim P_0^{(n-1)}+\delta_{n-2}=p_n+\delta_{n-2}\]
because the new paths in $P_0^{(n)}$ are the ones of length $n$.
In low degrees, we compute explicitly:
\begin{align*}
\delta_0&=\dim B_0=1=d_0=p_3,\\
\delta_1&=\dim P_0^{(1)}=1=d_1=p_4,\\
\delta_2&= p_2+\delta_0=0+1=1=d_2=p_5.
\end{align*}
For $n\geq 3$ we argue by induction:
\[ \delta_n=p_n+\delta_{n-2}=d_{n-3}+d_{n-2}=d_n\]
and on the other hand
\[ \delta_n=p_n+\delta_{n-2} =p_n+ p_{n+1}=p_{n+3}.\]
\end{proof}

\begin{rem}\label{rem:discussion_MZV}
We relate this to multiple zeta values. They are the periods of $\MTM$ which are real numbers.
Let $\MZV^{[0,n]}$ be the space of multiple zeta values of weight $0$ through $n$. Recall that $\zeta(2m)$ agrees with the period $(2\pi i)^{2m}$ of $\Q(-2m)$ up to a rational number, so it is a multiple zeta value of weight $2m$.
By the work of Brown, see \cite{brown-annals}, we have
\[ \Per(\MTM^{[0,n]})= \MZV^{[0,n]}\oplus 2\pi i\MZV^{[0,n-1]}.\]
Under the isomorphism
\[ B_n\tensor\C\isom \Pertilde(\MTM^{[0,n]})\tensor\C\]
the even part $B_n^\even$ corresponds precisely to the real periods and the odd part $B_n^\odd$  to the imaginary ones. We
obtain the same upper bound
\[ d_n\geq \dim\MZV^{[0,n]}\]
as predicted by Zagier and established  in the work  of Terasoma, Deligne--Goncharov.
\end{rem}

\begin{rem}
Note that $B_n^\even$ is a $B_n$-module and defines a Mixed Tate Motive. Its periods do \emph{not} agree with $\MZV^{[0,n]}$. In fact, at least conjecturally, $\MZV^{[0,n]}$ cannot be realised as a space of periods of a collection of 
motives. E.g., whenever the 
period $\zeta(3)$ appears, we expect to have $(2\pi i)^3$ as well.
\end{rem}

\begin{ex}
The quiver of $B_9$ has the following shape:

\begin{equation*}\begin{tikzpicture}[description/.style={fill=white,inner sep=2pt}]
   \matrix (m) [matrix of math nodes, row sep=2em,
                 column sep=2em, text height=1.5ex, text depth=0.25ex,
                 inner sep=2pt, nodes={inner xsep=0.3333em, inner
ysep=0.3333em}] at (0, 0)
    {  0 & 1 & 2 & 3 & 4 & 5 & 6 & 7 & 8 & 9 \\ %& 10 & 11 & 12 & 13 \\
    };
 \draw[ ->] ($(m-1-1.north) + (0.5mm,0mm)$) .. controls +(8mm,12mm) and
+(-8mm,12mm) .. ($(m-1-4.north) + (-0.5mm,0mm)$);

 \draw[ ->] ($(m-1-2.north) + (0.5mm,0mm)$) .. controls +(8mm,12mm) and
+(-8mm,12mm) .. ($(m-1-5.north) + (-0.5mm,0mm)$);

  \draw[ ->] ($(m-1-3.north) + (0.5mm,0mm)$) .. controls +(8mm,12mm) and
+(-8mm,12mm) .. ($(m-1-6.north) + (-0.5mm,0mm)$);

  \draw[ ->] ($(m-1-4.north) + (0.5mm,0mm)$) .. controls +(8mm,12mm) and
+(-8mm,12mm) .. ($(m-1-7.north) + (-0.5mm,0mm)$);

 \draw[ ->] ($(m-1-5.north) + (0.5mm,0mm)$) .. controls +(8mm,12mm) and
+(-8mm,12mm) .. ($(m-1-8.north) + (-0.5mm,0mm)$);

  \draw[ ->] ($(m-1-6.north) + (0.5mm,0mm)$) .. controls +(8mm,12mm) and
+(-8mm,12mm) .. ($(m-1-9.north) + (-0.5mm,0mm)$);

 \draw[ ->] ($(m-1-7.north) + (0.5mm,0mm)$) .. controls +(8mm,12mm) and
+(-8mm,12mm) .. ($(m-1-10.north) + (-0.5mm,0mm)$);

 \draw[ ->] ($(m-1-1.south) + (0.5mm,0mm)$) .. controls +(15mm,-27mm) and
+(-15mm,-27mm) .. ($(m-1-6.south) + (-0.5mm,0mm)$);

 \draw[ ->] ($(m-1-2.south) + (0.5mm,0mm)$) .. controls +(15mm,-27mm) and
+(-15mm,-27mm) .. ($(m-1-7.south) + (-0.5mm,0mm)$);

  \draw[ ->] ($(m-1-3.south) + (0.5mm,0mm)$) .. controls +(15mm,-27mm) and
+(-15mm,-27mm) .. ($(m-1-8.south) + (-0.5mm,0mm)$);

  \draw[ ->] ($(m-1-4.south) + (0.5mm,0mm)$) .. controls +(15mm,-27mm) and
+(-15mm,-27mm) .. ($(m-1-9.south) + (-0.5mm,0mm)$);

 \draw[ ->] ($(m-1-5.south) + (0.5mm,0mm)$) .. controls +(15mm,-27mm) and
+(-15mm,-27mm) .. ($(m-1-10.south) + (-0.5mm,0mm)$);

 \draw[ ->] ($(m-1-1.south) + (-0.5mm,0mm)$) .. controls +(30mm,-50mm) and 
+(-30mm,-50mm) .. ($(m-1-8.south) + (0.5mm,0mm)$);

 \draw[ ->] ($(m-1-2.south) + (-0.5mm,0mm)$) .. controls +(30mm,-50mm) and 
+(-30mm,-50mm) .. ($(m-1-9.south) + (0.5mm,0mm)$);

  \draw[ ->] ($(m-1-3.south) + (-0.5mm,0mm)$) .. controls +(30mm,-50mm) and 
+(-30mm,-50mm) .. ($(m-1-10.south) + (0.5mm,0mm)$);

 \draw[ ->] ($(m-1-1.north) + (-0.5mm,0mm)$) .. controls +(3mm,25mm) and 
+(-3mm,25mm) .. ($(m-1-10.north) + (0.5mm,0mm)$);

 \end{tikzpicture}\end{equation*}

By Proposition~\ref{prop:sum_paths}, the total dimension is 
\begin{multline*}
\dim_\Q B_9=10+9p_1+8p_2+7p_3+6p_4+5p_5+4p_6+3p_7+2p_8+p_9\\
=         10+0+0+7+0+5+4+3+4+2=35.
\end{multline*}
On the other hand, by Proposition~\ref{prop:zagier_fits} 
\[ \dim B_9^\even= \sum_{i=0}^9d_i=1+0+1+1+1+2+2+3+4+5=20\]
is the expected dimension for the space of $\MZV^{[0,9]}$.
Moreover, by \eqref{eq:odd}
\[ \dim B_9^\odd=\dim B_8^\even=15\]
is the expected dimension for the space $2\pi i\MZV^{[0,8]}\subset \Per(\MTM^{[0,9]})$. Finally
\[ \dim B_9^\even+\dim B_9^\odd=20+15=35,\]
confirming \eqref{eq:sum}. 
\end{ex}

\begin{ex} \label{Ex:NonHereditaryMTMsubcat}
Let $\Ah$ be the full subcategory closed under extensions in $\MTM_\Z$ generated by $\Q(0),\Q(3),\Q(6)$. Note that 
\[ \Ext^1(\Q(0),\Q(3))=\Ext^1(\Q(3),\Q(6))=\Q,\quad \Ext^1(\Q(0),\Q(6))=0.\]
The category is described by the quiver $B_2$ with three vertices $1,3,6$ and two edges $a$ from $1$ to $3$ and $b$ from $3$ to $6$.
Let $B$ be the quiver algebra. The category contains an object $M$ (motive) corresponding to $B'=B/\langle ab\rangle$. The category $\langle M\rangle$
is equivalent to the category of $B'$-modules, in particular not hereditary.
By Corollary~\ref{cor:MT_dim}  the space of periods of $M$ is at most $5$-dimensional over $\Qbar$. The Period Conjecture predicts that the dimension is equal to $5$.
\end{ex}

\subsection{Realisation of quiver algebras as motives}
Even the seemingly easy Mixed Tate Motives contain very complicated period spaces.

\begin{lemma}
Let $B=\Q Q/I$, where $Q$ is a quiver without oriented cycles and $I \subset \Q Q$ is a two-sided admissible ideal. 
%Let $Q$ be a quiver with relations without oriented cycles. Let $B$ be the corresponding basic algebra. 
Then there is a number field $F$ and a Mixed Tate Motive $M$ over $\Oh_F$ such that $\langle M\rangle $ is equivalent to the category of $B$-modules.
\end{lemma}
\begin{proof} Let $m$ be the maximal number of arrows between any two vertices in $Q$.

We choose a number field $F$ such that for all $i<j$
\[ \dim_\Q \Ext^1_{\MTM_{\Oh_F}}(\Q(i),\Q(j))\geq m\]
We can use $F=\Q(\mu_N)$ with $r_1=0$, $r_2=\phi(N)/2$ with $N$ big enough so that
$r_2> m$; or $F$ as in Example~\ref{ex:field} with $2^{n-1}(2^n-1)> m$. 

As $Q$ has no oriented cycles, we have a partial order on the set of vertices. We choose a total order compatible with the partial order, i.e., a numbering of the vertices such that $i<j$ whenever the vertices are comparable in the quiver. We identify $i$ with the Tate motive $\Q(i)$ and choose representatives
for the edges in the Ext-groups. This gives a description
of $B$ as quotient of the path algebra for a category of Mixed Tate Motives over $F$ (in general, by a non-admissible ideal -- we typically need to delete some arrows). Similar to Example \ref{Ex:NonHereditaryMTMsubcat}, the free module ${}_{B}B$ corresponds to an object $M \in \MTM_{\Oh_F}$.
\end{proof}

\begin{rem}\label{R:MixedTateArbitrary}
This means that arbitrarily complicated dimension formulas occur already for Mixed Tate Motives. In particular, the counterexamples in Example~\ref{Ex:dimEstimateLongerPaths} can be realised as Mixed Tate Motives.
\end{rem}

%% file: anhang.tex
\section{Proof of Proposition~\ref{prop:choice}}\label{app}

In this appendix, we give a proof of Proposition~\ref{prop:choice}.

Throughout the chapter, we fix a perfect field $k$ and a category $\Ch$ as in Set-Up~\ref{setup:artin}. We consider the system of all strongly finitary full abelian subcategories $\Bh$ closed under subquotients. 

By Proposition~\ref{P:Deligne} and Morita theory (Lemma~\ref{lem:Morita}) such a $\Bh$ is equivalent to the category of finitely generated modules over a basic finite dimensional algebra $B$. The assumption on the partial order of the simple objects of $\Ch$ implies that the species of $\Bh$ is contains no oriented cycles. We first show that this has strong consequences for projective covers of simple objects in $\Bh$.

\begin{lemma}\label{lem:hull_unique}
Let $B$ be a basic finite dimensional algebra whose species contains no oriented cycles. 
Then the projective cover $P(S)$ of a simple object $S$ is unique up to unique isomorphism in $\modules{B}$.
\end{lemma}
\begin{proof} Projective covers exist and are unique up to non-unique isomorphism in general. Since the species contains no oriented cycles, there is an isomorphism
\begin{align}\label{E:IsoProjSimp}
 \End_B(P(S)) \cong \End_B(S),
 \end{align}
so that the identity of $S$ has a unique lift to $P(S)$. This implies uniqueness of the isomorphism.

The argument for \eqref{E:IsoProjSimp} is analogous to the case of path algebras of quivers.
\end{proof}

\begin{lemma}\label{lem:canonical}
The equivalences  $\Bh\isom \modules{B}$ can be chosen in a compatible way, i.e., such that the inclusions $\Bh\subset\Bh'$ are induced by functorial epimorphisms $B'\to B$.
\end{lemma}
\begin{proof}
For every isomorphism class of simple objects in $\Ch$ we fix a representative. 
Note that $\Bh\subset\Ch$ is closed under isomorphisms. If it contains a simple object isomorphic to such a representative, it also contains the representative itself.

For a simple object $S$ in $\Bh$, let $P_\Bh(S)$ be its projective cover in $\Bh$. We set
\[ B:=\End_\Ch\left( \bigoplus_S P_\Bh(S)\right)\]
where the sum is taken over the fixed simple objects of $\Bh\subset\Ch$.

Given an inclusion $\Bh\subset\Bh'$, let $\overline{P_{\Bh'}(S)}$ be the maximal quotient of $P_{\Bh'}(S)$ which is contained in $\Bh$. We want to show that it exists. We will find $\overline{P_{\Bh'}(S)}$ (more precisely, $P_{\Bh'}(S)\twoheadrightarrow \overline{P_{\Bh'}(S)}$) as the projective limit of the system of all quotients $P_{\Bh'}(S)\twoheadrightarrow Q$ with $Q$ in $\Bh$. This system is filtered: any two such quotients $Q_1$ and $Q_2$ are dominated by the image of
$P_{\Bh'}(S)$ in $Q_1\times Q_2$.
It is in $\Bh$ because $\Bh$ is closed under subobjects. The projective limit exists because $P_{\Bh'}(S)$ has finite length.

Let $P_{\Bh'}(S)\to M$ be a morphism towards an object in $\Bh$. We claim that it factors canonically via $\overline{P_{\Bh'}(S)}$. As $\Bh$ is closed under subobjects, the image $Q\subset M$ is in
$\Bh$ again. By definition of the projective limit, the map
$P_{\Bh'}(S)\twoheadrightarrow Q$ factors canonically via $P_{\Bh'}(S)\to \overline{P_{\Bh'}(S)}$.

Hence the quotient $\overline{P_{\Bh'}(S)}$ is projective in
$\Bh$. As a quotient of a local module it is indecomposable, hence a projective cover of $S$ in $\Bh$. 
By Lemma~\ref{lem:hull_unique} it is canonically isomorphic to $P_\Bh(S)$. 
Every endomorphism of $\bigoplus_SP_{\Bh'}(S)$ induces a unique endomorphism
of $\bigoplus_S\overline{P_{\Bh'}(S)}$. This defines
\[ \phi_{B'B}:B'\twoheadrightarrow B,\]
which is surjective because $\bigoplus_S P_{\Bh'}(S)$ is projective in $\Bh'$.
\end{proof}

\begin{rem}The ring $\tilde{B}:=\varprojlim_{\Bh}B$ is pseudo-compact in the sense of \cite{gabriel}. We find the category $\Ch$ as its discrete representations. We do not use this fact.
\end{rem}

By Proposition~\ref{P:SpeciesWithRel} (i.e., \cite[Theorem~3.12]{Berg}) there is an isomorphism 
\[B\isom k\cs_B/\Ih\] for some two-sided ideal $\Ih$  where $\cs_B$ is the species of
$\modules{B}$. 
Recall from Definition~\ref{defn:species_of_c} and Definition~\ref{defn:q_from_cs} that  
\begin{equation}\label{eq:praesentation}  k\Sh_B=T_K(D\Ext^1_\Bh(K,K))\end{equation}
 is the tensor algebra over $K=B/\rad(B)$ and note that $\Ext^1_\Bh(X,Y)\subset \Ext^1_\Ch(X,Y)$ for all $X,Y\in\Bh$. 
The species $\cs_B$ of $\Bh$ is
directed by our condition on $\Ch$, hence $k\cs_B$ is finite dimensional. 
By
Theorem~\ref{thm:generate_ideal} the ideal $\Ih$ is generated by $D\Ext^2_\Bh(K,K)$.
This establishes the characterisation of $\Bh$ as claimed in Proposition~\ref{prop:choice}.

%The above construction depends on choices. We go through the arguments in \cite[Theorem~3.12]{Berg} in order to make them transparent: 

We will first go through %the steps in %{\cite[Theorem~3.12]{Berg} 
the proof of Proposition~\ref{P:SpeciesWithRel} in order to
collect all instances where choices are being made. We will then show how to make these choices in a compatible  (not canonical) way using Zorn's Lemma.

\emph{1. Step:}
By the Theorem of of Wedderburn and Malcev, see \cite[11.6~Theorem]{pierce}, the projection $B\to K$ has a section 
\begin{equation}\label{eq:sp1} K\to B\end{equation}
by an algebra homomorphism. It is unique up to conjugation by $(1-w)$ for
 $w\in\rr_B=\rad(B)$. The choice of \eqref{eq:sp1} turns $B$ into a 
$K-K$-bimodule. 

\emph{2. Step:}
%The long exact sequence attached to the short exact sequence
%\begin{equation}\label{eq:rr} 0\to \rr_B\to B\to K\to 0\end{equation}
%and the functor $\Hom_B(\cdot,K)$ induces a natural isomorphism
%\[ \Ext^1_B(K,K)\isom \Hom_B(\rr_B,K)\isom \Hom_K(\rr_B/\rr_B^2,K)\]
%because $\Hom_B(B,K)=\Hom_B(K,K)$ and $\Ext^1_B(B,K)=0$.
%By Lemma~\ref{lem:Ext_Tor} we deduce
%\[ \Ext^1_B(K,K)\isom D\rr_B/\rr_B^2.\]
%Dually we get
By Lemma~\ref{lem:compute_ext} there is a canonical isomorphism of
$K-K$-bimodules
\begin{equation}\label{canApp} D\Ext^1_B(K,K)\isom \rr_B/\rr_B^2.\end{equation}
%(Note that this is a special case of  Theorem~\ref{T:Bongartz}. However, its proof depends on the result which we are reviewing).
%The sequence \eqref{eq:rr} is a sequence of $B-B$-bimodules. Via \eqref{eq:sp1} we consider it as a sequence of $K-K$-bimodules.
We choose a  splitting $\rr_B/\rr_B^2\to \rr_B$ as $K-K$-bimodules. It exists because $K\tensor_k K\op$ is semi-simple because $k$ is perfect. Together with \eqref{canApp} this yields
\begin{equation}\label{eq:sp2} D\Ext^1_B(K,K)\to\rr_B.\end{equation}

\emph{3. Step:}
By construction \eqref{eq:praesentation}
the maps \eqref{eq:sp1} and \eqref{eq:sp2} induce an algebra epimorphism
\[ \pi: k\cs_B\to B.\]
It induces an isomorphism $\rr_{k\cs_B}/\rr_{k\cs_B}^2\isom \rr_B/\rr_B^2$, hence the kernel $\ci_B$ is contained in $\rr_{k\cs_B}^2$. As $\pi$ is a morphism of $K-K$-bimodules, $\ci_B$ has a natural structure of $K-K$-bimodule as well.

\emph{4. Step:}
By Theorem~\ref{thm:generate_ideal}, we have a canonical isomorphism
\[ D\Ext^2_{B}(K,K)\isom \Ih_B/(\Ih_B\rr_{k\cs_B}+ \rr_{k\cs_B}\Ih_B).\]
As in the 2. step, we choose a splitting  of
\begin{equation}\label{eq:step4}  \Ih_B\to \Ih_B/(\Ih_B\rr_{k\cs_B}+ \rr_{k\cs_B}\Ih_B)\end{equation}
in the categeory of $K-K$-bimodules. This yields
\begin{equation}\label{eq:sp3} D\Ext^2_B(K,K)\hookrightarrow \Ih_B.\end{equation}

%%%%%%%%%%%%%%%%%%%%%%%%%%%%%%%%%%%%%%%%%%%%%%%%%%%%%%%%%%
\begin{lemma}\label{lem:choice1}
Let $\Ch$ be as in Set-Up~\ref{setup:artin} and let $\Bh_1\subset\Bh_2\subset\Ch$ be strongly finitary full subcategories closed under subquotients. Given $\Ih_1\subset k\cs_1$ and an inclusion of $D\Ext^2_{\Bh_1}(K_1,K_1)\subset \Ih_1$ splitting the canonical \eqref{eq:step4}, we can choose $\Ih_2\subset k\cs_2$ and an inclusion of $D\Ext^2_{\Bh_2}(K_2,K_2)\subset\Ih_2$ splitting the canonical \eqref{eq:step4} such that the two obvious diagrams commute.
\end{lemma}
\begin{proof}
Let $B_1$ and $B_2$ be the basic algebras of Lemma~\ref{lem:canonical}. As they are canonical, the inclusion of categories is induced by a surjection $B_2\twoheadrightarrow B_1$. Since $B_2$ is finite dimensional, it induces a surjection $\rr_{B_2}\twoheadrightarrow \rr_{B_1}$. By assumption, $\Bh_1\isom \modules{k\cs_1/\Ih_1}$, hence $B_1\isom k\cs_1/\Ih_1$. In particular, we are given an epimorphism
\[ k\cs_1\to B_1\]
defining $K_1\to B_1$ and $D\Ext^1_{\Bh_1}(K_1,K_1)\isom \rr_{B_1}/\rr_{B_1}^2$.

\emph{1. step (choice of $K_2\to B_2$):} We are given a commutative diagram
\[\begin{xy}\xymatrix{
B_2\ar[r]\ar[d]&K_2\ar[d]\\
B_1\ar[r]&K_1
}\end{xy}\]
and a splitting $\rho_1:K_1\to B_1$ as $k$-algebras. 
We recurr to the construction of $B_2$. Let $\Sigma_2$ be a set of representatives for the isomorphism classes of simple objects in $\Bh_2$ and $\Sigma_1\subset \Sigma_2$ the subset of those in $\Bh_1$. Let $\Sigma'=\Sigma_2\ohne \Sigma_1$ be the complement. We put 
\[ B_{12}=\End\left( \bigoplus_{S\in\Sigma_1}P_{\Bh_2}(S)\right),\quad B'=\End\left( \bigoplus_{S\in\Sigma'}P_{\Bh_2}(S)\right).\]
Note that the radical of $B_{12}$ is $K_1$ and the radical of $B_{12}\times B'$ is $K_2\isom K_1\times K'$. We have a commutative diagram of unital $k$-algebras
\[\begin{xy}\xymatrix{
B_{12}\times B'\ar[r]\ar[d]& B_2\ar[r]\ar[d]&K_2\ar[d]\\
B_{12}\ar[r]  &B_1\ar[r]&K_1
}\end{xy}\]
Choose splittings $\rho':K'\to B'$ and $\rho_{12}:K_1\to B_{12}$.
%\[ \begin{xy}\xymatrix{
%  K_1\ar[r]^{\rho_2}\ar[d]^{=}&B_{12}\ar[d]\\
%  K_1\ar[r]^{\rho_1}&B_1
%}\end{xy}\]
The diagram
\[ \begin{xy}\xymatrix{
  &B_{12}\ar[d]\\
  K_1\ar[r]^{\rho_1}\ar[ru]^{\rho_{12}}&B_1
}\end{xy}\]
defines another splitting of $B_1\to K_1$. It agrees with $\rho_1$ up to conjugation with $(1-w_1)$ for $w_1\in\rr_{B_1}$. Let
$w_2\in\rr_{B_{12}}$ be a preimage of $w_1$. We replace $\rho_{12}$ by its conjugation by $(1-w_2)$. This makes the splittings $\rho_1$ and $\rho_{12}$ compatible. Then
\[ K_2\isom K_{1}\times K'\xrightarrow{(\rho_{12},\rho')} B_{12}\times B'\to B_2\]
has the required property.

\emph{2. step:} We are given
a section of $K_1-K_1$-bimodules
\[ \rr_{B_1}/\rr_{B_1}^2\to \rr_{B_1}.\]
By semi-simplicity, it can be lifted to a section of $K_2-K_2$-bimodules
\[\rr_{B_2}/\rr_{B_2}^2\to \rr_{B_2}.\]
In other words, the diagram
\[\begin{xy}\xymatrix{
\rr_{B_2}/\rr^2_{B_2}\ar[r]\ar@{->>}[d] &\rr_{B_2}\ar@{->>}[d]\\
\rr_{B_1}/\rr^2_{B_1}\ar[r]&\rr_{B_1}
}\end{xy}\]
commutes.

\emph{3. step (choice of $\ci_2$):} 
The compatible choices in Step 1 and 2 define
compatible surjections
\[ k\cs_i\twoheadrightarrow B_i.\]
It induces a map on kernels $\Ih_2\to\Ih_1$.

\emph{4. step (choice of generators of $\ci_2$):}
We are given a splitting of
\[  \Ih_1\to \Ih_1/(\Ih_1\rr_{k\cs_1}+ \rr_{k\cs_1}\Ih_1)\]
in the categeory of $K_1-K_1$-bimodules.
By semi-simplicity the choice for $\ci_2$ can by made in a compatible way with a given choice for $\ci_1$.
\end{proof}

\begin{lemma}\label{lem:choice2}
Let $\Ch$ be as in Set-Up~\ref{setup:artin}, $\Bh_1,\Bh_2\subset\Ch$ a strongly finitary full subcategories closed under subquotients. Let
$\Bh_0=\Bh_1\cap\Bh_2$ and $\tilde{\Bh}$ the smallest full abelian category containing $\Bh_1$ and $\Bh_2$.

Given $\Ih_j\subset k\cs_j$ and an inclusion of $D\Ext^2_{\Bh_j}(K_j,K_j)\subset \Ih_j$ for $j=0,1,2$ compatible with respect to the inclusion $\Bh_0\subset\Bh_1,\Bh_2$, then there are canonical inclusions $\tilde{\Ih}\subset k\cs_{\tilde{B}}$ and
$D\Ext^2_{\tilde{\Bh}}(\tilde{K},\tilde{K})\subset\tilde{\Ih}$ compatible with the inclusions for $\Bh_1$ and $\Bh_2$.
\end{lemma}
\begin{proof}
Let $B_0,B_1,B_2,\tilde{B}$ be the corresponding basic algebras with semi-simple quotients $K_0, K_1,K_2,\tilde{K}$. A simple object is in $\tilde{\Bh}$ if and only if it is in $\Bh_1$ or $\Bh_2$. The endomorphism algebras agree because they are computed in $\Ch$. Hence
\[ \tilde{K}\isom K_1\times_{K_0} K_2.\]
For simple objects $S,T$, we have
\[ \Ext^1_{\tilde{\Bh}}(S,T))=\Ext^1_{\Bh_1}(S,T)+\Ext^1_{\Bh_2}(S,T)\subset \Ext^1_{\Ch}(S,T)\]
where we interpret $\Ext^1_{\Bh_j}(S,T)=0$ is $S$ or $T$ is not in $\Bh_j$. Hence
\[ \Ext^1_{\tilde{\Bh}}(S,T)\isom (\Ext^1_{\Bh_1}(S,T)\oplus\Ext^1_{\Bh_2}(S,T))/\Ext^1_{\Bh_0}(S,t)\]
and 
\[  D\Ext^1_{\tilde{\Bh}}(S,T)\isom D\Ext^1_{\Bh_1}(S,T)\times_{D\Ext^1_{\Bh_0}(S,T)}D\Ext^1_{\Bh_2}(S,T).\]
This implies
\[ k\cs_{\tilde{\Bh}}\isom k\cs_{\Bh_1}\times_{k\cs_{\Bh_0}} k\cs_{\Bh_2}.\]
For $j=1,2$ let $J_j$ be the kernel of 
\[ k\cs_{\tilde{\Bh}}\to k\cs_{\Bh_j}\to B_j.\]
We put $\tilde{I}=J_1\cap J_2$. Then $k\cs_{\tilde{B}}/\tilde{I}$ maps
to $B_1$ and to $B_2$ and is the minimal quotient of the path algebra with this property. Hence
\[ k\cs_{\tilde{B}}/\tilde{I}\isom\tilde{B}.\]
By construction $\tilde{I}$ maps to the ideals $I_j\subset k\cs_{\Bh_j}$.
This settles the compatibility issue for the ideals.

We are given compatible maps
\[ D\Ext^2_{\Bh_j}(K_j,K_j)\to I_j\]
for $j=0,1,2$.
They induce
\begin{multline*}
 D\Ext^2_{\tilde{\Bh}}(\tilde{K},\tilde{K})\to
D\Ext^2_{\Bh_1}(K_1,K_1)\times_{D\Ext^2_{\Bh_0}(K_0,K_0)}D\Ext^2_{\Bh_2}(K_2,K_2)\\
\to I_1\times_{I_0}I_2\to k\cs_{\Bh_1}\times_{k\cs_{\Bh_0}} k\cs_{\Bh_2}
\to k\cs_{\tilde{\Bh}}
\end{multline*}
which takes image in $\tilde{I}$.
\end{proof}

\begin{proof}[Proof of Proposition~\ref{prop:choice}.]
Since $\Ch$ was assumed essentially small, its full abelian subcategories closed under subquotients (and hence isomorphisms) form a set.

We use Zorn's Lemma. We consider the set of triples 
\[ 
\left(\Ch', \{I_\Bh\subset k\cs_\Bh\}_{\Bh\subset\Ch'}, \{D\Ext^2_{\Bh}(K_\Bh,K_\Bh)\subset\Ih_\Bh\}_{\Bh\subset\Ch'}\right)
\]
consisting of a full abelian subcategory $\Ch'\subset\Ch$ closed under subquotients and a compatible system of choices of $\Ih_\Bh\subset k\cs_\Bh$ such that
$\Bh\isom \modules{k\cs_\Bh/\Ih_\Bh}$
and $D\Ext^2_{\Bh}(K_\Bh,K_\Bh)\subset \Ih_\Bh$ generating the ideal for all strongly finitary full subcategories $\Bh\subset\Ch'$ closed under subquotients.  The set is non-empty because a choice for a strongly finitary $\Ch'$ also induces a choice for all its subcategories.
 
We define a partial order by inclusion of subcategories and compatibility of the choices. Given a totally ordered subset, we obtain an upper bound by taking the union of the subcatgories. By Zorn's Lemma there is a maximal triple. Let $\Ch'$ be the category in this maximal triple. 

If $\Ch'$ is not equal to $\Ch$, then there is
an object $M\notin\Ch'$. Let $\Bh_2=\langle M\rangle$ and
$\Bh_1=\Bh_2\cap\Ch'$. Both are strongly finitary. As $\Bh_1\subset\Ch'$, we have a prefered choice of data. By Lemma~\ref{lem:choice1}, we find a compatible choice of data for $\Bh_2$. 

Let $\Ch''$ be the full abelian subcategory of $\Ch$ generated by $\Ch'$ and $\Bh_2$. We claim that the data extends to all strongly finitary subcategories $\Bh$ of $\Ch''$. Indeed, $\Bh$ is generated by $\Bh\cap \Bh_2$ and
$\Bh\cap \Ch''$, where we have compatible choices of data. By Lemma~\ref{lem:choice2} the data extend canonically to $\Bh$.

This would contradict maximality of $\Ch'$, hence $M$ cannot exist.
\end{proof}